\documentclass[a4paper,11pt]{article}
\usepackage{geometry} 
\newgeometry{vmargin={30mm}, hmargin={22mm,22mm}} 
\usepackage{graphicx} 
\usepackage{authblk}
\usepackage{amsmath}
\usepackage{amssymb}
\usepackage{amsfonts}
\usepackage{amsthm}
\usepackage{xcolor}
\usepackage{hyperref}
\usepackage{caption}  
\usepackage{subcaption} 
\usepackage{moreverb}  
\usepackage{hyperref} 
\usepackage{mathrsfs} 
\usepackage{epstopdf} 
\usepackage{psfrag} 
\usepackage{enumitem} 
\usepackage{graphics}
\usepackage{graphicx} 
\usepackage{float} 
\usepackage{afterpage}
\usepackage{lineno} 
\usepackage{tabularx} 
\usepackage{mathtools}
\usepackage{physics}
\usepackage{float}
\usepackage{eucal} 
\allowdisplaybreaks 

\usepackage{hyperref}
\hypersetup{
    colorlinks= true,
    linkcolor=blue,
    filecolor=blue,
    urlcolor=blue,
    citecolor=blue,
}

\usepackage[numbers]{natbib}  

\usepackage{circuitikz} 
\usepackage{tikz}
\usetikzlibrary{positioning}
\definecolor {processblue}{cmyk}{0.96,0,0,0}

\newtheorem{Theorem}{Theorem}[section]

\newtheorem{Remark}[Theorem]{Remark}

\title{Discrete-Time Optimal Control of Species Augmentation for Predator-Prey Model}
\author[1]{Munkaila Dasumani\thanks{Corresponding author: munkaila5@gmail.com}}
\author[2]{Suzanne Lenhart}
\author[3]{Gladys K. Onyambu}
\author[4]{Stephen E. Moore}
\affil[1]{Department of Mathematics, Institute for Basic Sciences, Technology and Innovation, Pan African University, Nairobi, Kenya}
\affil[2]{Department of Mathematics, University of Tennessee, Knoxville, Tennessee, USA}
\affil[3]{Department of Zoology, Jomo Kenyatta University of Agriculture and Technology, Nairobi, Kenya}
\affil[4]{Department of Mathematics, University of Cape Coast, Ghana}
\date{} 
\begin{document} 
	\maketitle
\section*{Abstract}
 Species augmentation is one of the methods used to promote biodiversity and prevent endangered species loss and extinction. The current work applies discrete-time optimal control theory to two models of species augmentation for predator-prey relationships. In discrete-time models, the order in which events occur can give different qualitative results. Two models representing different orders of events of optimal augmentation timing are considered. In one model,  the population grows and predator-prey action occurs before the translocation of reserve species for augmentation. In the second model, the augmentation happens first and is followed by growth and then predator-prey action. The reserve and target populations are subjected to strong Allee effects. The optimal augmentation models employed in this work aim to maximize the prey (target population) and reserve population at the final time and minimize the associated cost at each time step. Numerical simulations in the two models are conducted using the discrete version of the forward-backward sweep method and the sequential quadratic programming iterative method, respectively. The simulation results show different population levels in the two models under varying parameter scenarios. Objective functional values showing percentage increases with optimal controls are calculated for each simulation. Different optimal augmentation strategies for the two orders of events are discussed.  This work represents the first optimal augmentation results for models incorporating the predator-prey relationship with discrete events.

 \subsection*{ Keywords:}   
  Discrete-time models;  Forward-backward sweep method; Order of events; Predator-prey interaction; Sequential quadratic programming; Species augmentation


\section{Introduction}
The interaction between predators and prey is crucial in many ecological systems. These interactions can have positive or negative repercussions on both species. For instance, with a global population size of less than 500 individuals, the hirola antelope,  {\it Beatragus hunteri}, primarily found in northeastern Kenya, is the world’s most endangered/threatened antelope. Predation and habitat loss are thought to be responsible for low abundances of the hirola antelope population \cite{ali2017resource}. Lions, cheetahs, and African wild dogs are their primary predators. However, the battle to save the hirola antelope from extinction is a race against time, which needs critical augmentation strategies \cite{ali2017resource}. Other possible endangered/target species affected by predator-prey interactions that need critical augmentation include the woodland caribou ({\it Rangifer tarandus caribou}), in British Columbia \cite{wittmer2005role,wittmer2005population},  and the Australian terrestrial mammal fauna such as eastern betting ({\it Bettongia gaimardi}), and eastern barred bandicoot ({\it Perameles gunnii}), that has suﬀered a very high rate of decline and extinction relative to predation \cite{radford2018degrees}.

 The movement of animals to re-establish extirpated populations or to augment declining populations is an important tool for conservation biologists~\cite{seddon2014reversing}.  Species augmentation, also referred to as stocking reintroduction \cite{cheyne2006wildlife}, is one of the methods used to prevent species loss and extinction. Several population studies on some threatened/endangered species have recommended the use of augmentation as a technique to prevent species extinction and safeguard biodiversity. Some of these studies include the optimal timing of augmentation of a threatened species in a target region by moving individuals from a reserve or captive population to augment a declining population \cite{bodine2008optimal,bodine2012order}. A rescue of an endangered carnivore using an augmentation strategy was investigated by Manlick et al. \cite{manlick2017augmentation}.  In addition, the optimal genetic augmentation strategies for a threatened species using a continent-island model have been studied by Bodine and  Martinez \cite{bodine2014optimal}. A model for the optimal translocation of an age-structured black rhino population that compares strategies for maximizing the translocation rate and the growth of a newly formed population was presented by Hearne and Swart \cite{hearne1991optimal}. Optimal control of species augmentation in a competition model of differential equations was investigated in \cite{dasumani2025optimal}, where two objective functionals involving the populations and the cost of the controls were considered. The successful reintroduction of African wild dogs ({\it Lycaon pictus}) to Gorongosa National Park in Mozambique has been described in \cite{bouley2021successful}. Their work represents the first transboundary translocation and reintroduction of founding packs of wild dogs to Gorongosa over a 28-month study period.  More works on species augmentation and translocation of endangered species can be found in \cite{haines2006habitat, hearne1991optimal,kingston2004conservation, sinclair2005dynamics, maes2004functional,romain2004density, rasmussen2009endangered}.

   Several developments of the Lotka \cite{lotka1925elements} and Volterra \cite{volterra1926fluctuations} systems have been explored since they were first proposed in the 1920s. Some of these developments using differential equation models can be found in the works \cite{bodine2017predator, aulisa2014continuous, mandal2024dynamic, flores2014dynamics, sasmal2015intra, xue2024analysis,biswas2022dynamics,vijayalakshmi2024mittag, dasumani2024nonlinear,yao2022bifurcations} and the references therein. In addition to the above extension of the Lotka-Volterra system, optimal control of an infected prey-predator model was studied by Qiu et al. \cite{qiu2024optimal}.  The work \cite{kar2012sustainability} provides optimal control of an exploited prey-predator system through the provision of alternative food to predators. San G. et al \cite{san1974optimal} have explored an optimal control of a prey-predator system.   Ibrahim \cite{ibrahim2021optimal} has studied the optimal harvesting of a predator-prey system with a marine reserve.
   
   An optimal control for a predator-prey model with disease in the prey population is studied in \cite{simon2018optimal}. A discrete time-optimal control of an additional food provided to predator-prey systems with applications to pest management and biological conservation is studied in \cite{srinivasu2010time}.  Modeling infectious disease and prey-predator interaction with optimal control theory is studied in \cite{mekonen2024mathematical}. Lazaar and Serhani \cite{lazaar2023stability} have studied the optimal control of a prey-predator model with prey refuge and prey infection. Additional applications of optimal control theory of Lotka-Volterra models can be found in \cite{chatterjee2023predator,chang2012hopf,mondal2023autonomous,miao2024parameter,juneja2023dynamical,ghosh2014sustainable,ibrahim2021optimal}. 

   Mathematical models could be represented as continuous-time or discrete-time equations. Continuous-time models of differential equations have been used to model overlapping generations in a population, assuming that all events denoting biological mechanisms occur simultaneously. However, many species with non-overlapping generations in their populations have well-defined cycles of reproduction (births and deaths are not evenly distributed over time) that generally occur over a few weeks or months and are suitably modeled using discrete-time difference equations.
   
  Discrete-time models have received significant contributions in recent years due to their ability to depict real-world scenarios, such as the model defined in this current work. Discrete-time difference equations have been successfully applied to several population biology and behavioral ecological scenarios \cite{clark2000dynamic,hof1998spatial,hof2002spatial,el2020spatiotemporal,dasumani2025augmentation}. In addition, discrete-time optimal control applied to pest control problems is investigated by Ding et al. \cite{ding2014discrete}. Optimal control of harvest timing in discrete population models was provided by Grey et al. \cite{grey2021optimal}. Whittle et al. \cite{whittle2006optimal} have studied optimal control for managing an invasive plant species. They formulated and solved a discrete-time optimal control problem to determine where control is best applied over a finite time horizon. 
  
  The sequence in which events occur is important in discrete-time models. An overview of the timing of events and model formulation for discrete-time difference equation models is outlined in Section $2.4.2$ of
Caswell's book \cite{caswell2000matrix}. The order of events in optimal control of the integrodifference model is investigated in \cite{lenhart2009investigating}. Again, optimal control of the integrodifference equation with growth-harvesting-dispersal order was presented by Zhong and Lenhart \cite{zhong2012optimal}.  Furthermore, comparing discrete models for optimal control of species augmentation with different possible orders of events was examined by Bodine et al. \cite{bodine2012order}.

  In this current study, we provide a novel discrete‐time augmentation model for predator-prey dynamics with two orders of events: {\bf Model A}: growth followed by predator-prey action and then augmentation,  and {\bf Model B}: augmentation followed by growth and then predator-prey action.  We assume in this work that the prey population is on the verge of declining due to predation; hence, the prey population is referred to as the endangered/targeted species.  Note that one may want to consider the predator population as an endangered/target species if initially their population level is below the threshold for growth. The model presented in this study can be applied to predator-prey interaction models, such as the case of the hirola antelope. This work marks the first study of optimal augmentation control in discrete-time models with predator-prey relationships. The objective functional seeks to maximize the prey (target population) and the reserve population at the final time and minimize the associated augmentation cost at each time step. 

  The rest of the paper is structured as follows:  Section~\ref{sec2a} shows the discrete model without augmentation. Section~\ref{sec2} presents the formulation of the augmentation models with the order of events. Results and discussions of the models, including numerical simulations, are presented in Section~\ref{sec3}. The conclusion is captured in Section~\ref{sec4}.

\section{Formulation of the discrete model with augmentation}
\label{sec2a}
In this section, we consider the formulation of a discrete-time predator-prey model with three populations, also referred to as states; the prey population  {\bf u}, predator population  {\bf v}, and reserve population  {\bf w}  are denoted by the vector 
\begin{equation} \label{vectoreqn}
    {\bf u} = \left(u_0, u_1, \dots,u_T\right), \hspace{10pt}  {\bf v}  = \left(v_0, v_1, \dots, v_T\right), \hspace{10pt}   {\bf w}  = \left(w_0, w_1, \dots, w_T\right),
\end{equation} 
respectively, and the subscripts denote the time steps. We assume in this work that the prey population is declining due to predation; hence, the prey population is referred to as the endangered/targeted species. For the order of events here, the growth of  {\bf u}  and  {\bf w}  populations comes first, and then the predator-prey action happens in  {\bf u}  and   {\bf v}. Thus, the prey population grows, followed by predator-prey interaction, and the predator decay.  Without an augmentation effort, our model is defined by the following nonlinear discrete difference equations

\begin{equation}\label{predatoreqn1}
    \left\{\begin{aligned}
     & u_{t+1}=u_t \left[s\left(1-\dfrac{u_t}{k_u}\right)\left(\dfrac{u_t}{k_u}-m\right)+1 \right] (1- \delta_1 v_t), \\
     &v_{t+1}= \left(v_t+u_t \left[s\left(1-\dfrac{u_t}{k_u}\right)\left(\dfrac{u_t}{k_u}-m\right)+1 \right] 
     \delta_2 v_t \right) (1- \gamma),\\
        &w_{t+1}=w_t + qw_t\left(1-\dfrac{w_t}{k_w}\right)\left(\dfrac{w_t}{k_w}-n\right),\\
 \end{aligned}\right.
\end{equation}
for each time step $t$ $ (t = 0, 1, \dots, T-1)$,  with given initial population levels, 
\begin{equation}\label{initialdiscrete}
   {\bf u}_0,   {\bf v}_0,  {\bf w}_0,
\end{equation}
where $u_t$ is the prey population, $v_t$ is the predator population, and $w_t$ reserve population at time step $t$. The parameters $s,k_u,m,,\delta_1, \delta_2, \gamma, q, k_w, n$ are all positive constants.  The parameter $s$ denotes the intrinsic growth rate of the prey, $q$ is the intrinsic growth rate of the reserve population, $k_u$ is the carrying capacity of the prey, and $k_w$ is the carrying capacity of the reserve. With  $m$ and $n$ as strong Allee effect constants,  then  $m k_u$ and $n k_w$  are the Allee thresholds for growth 
for the prey and the predator populations, respectively, such that $0<m<1$ and $0<n<1$  
  \cite{courchamp1999inverse, maciel2015allee, bodine2008optimal,musgrave2015population}. The parameter $\delta_1$ represents the consumption rate of prey by a predator, $\delta_2$ represents the attack rate, and $\gamma$ represents the natural decay of  {\bf v}   (with or without  {\bf u} ). The reserve population is a viable source for harvesting individuals to augment the endangered population. The reserve and the target populations are assumed to have strong Allee effect growth.

\section{Order of events}
\label{sec2}
The order of events in discrete-time models significantly impacts the dynamics of the species. According to Bodine et al. \cite{bodine2012order}, the optimal timing of augmentation of a threatened/endangered species in a target zone must all be converted to a certain order in the species' life cycle. 
We will develop two discrete-time augmentation models representing different sequences of events. In the current study, the two main orders of events  employed are:
 \begin{itemize}
    \item {\bf Model A}: Growth followed by predator-prey action and  then augmentation;
    \item {\bf Model B}: Augmentation followed by growth and then predator-prey action.
 \end{itemize}
 
 Now let $f_{\bf u}(u_t)$, $f_{\bf v}(v_t)$, and $f_{\bf w}(w_t)$ be functions of the growth or decay of the prey, predator, and the reserve populations, respectively,
 \begin{equation*}
    \begin{aligned}
      f_{\bf u}(u_t) & = su_t\left(1-\dfrac{u_t}{k_u}\right)\left(\dfrac{u_t}{k_u}-m\right)+u_t, \\
      f_{\bf v}(v_t) & =  (1- \gamma)v_t,\\
         f_{\bf w}(w_t) & = qw_t\left(1-\dfrac{w_t}{k_w}\right)\left(\dfrac{w_t}{k_w}-n\right) + w_t.\\
 \end{aligned}
\end{equation*}
Then the order of events of the model without the augmentation can be written as 
  \begin{equation*}\label{systemR}
    \left\{\begin{aligned}
      u_{t+1} & = f_{\bf u}(u_t) - \delta_1 f_{\bf u}(u_t)v_t, \\
     v_{t+1} & = f_{\bf v}\left(v_t + \delta_2f_{\bf u}(u_t)v_t \right),\\
     w_{t+1} & = f_{\bf w}(w_t).
 \end{aligned}\right.
\end{equation*}
In this order of events with no augmentation:  (I) the prey population grows, (II) the predator-prey interaction happens, and (III) the predator decays. Note that the reserve population is decoupled and does not factor into the order of events. 

In the case when the above order of events happened in the habitat of the species, it can cause the extinction of both the predator and the prey populations, thereby creating an imbalance in the ecological system. Hence, this motivated the formulation of the optimal control models based on the two orders of events described in this work. However, one can also explore other possible orders of events that can emanate from our model.

We study how the different strategies employed in the order of events will impact the optimal augmentation results. The species dynamics of the reserve, target, and predator-prey interaction will also be investigated. The augmentation controls are chosen to maximize the prey (target population) and the reserve population at the final time and minimize the associated cost at each time step $(0, 1, 2, \dots, T-1).$   We assume to  maximize  the population $(\bf u + w)$ by final time $T$  with  proportional weights applied to both species. However, it is important to maximize the target population than the reserve population at the final time. The vector of controls ${\bf h} = (h_0, h_1, \dots, h_{T-1})$, where $h_t$ represents the control effort to move the proportion of the reserve population to augment the target population at each time step $t$ in both models.

     \subsection{{\bf Model A}: Growth followed by predator-prey action and  then augmentation}
     \label{sec 4.1}
In this model, the reserve population grows followed by the predator-prey interaction and then translocating individuals (from  {\bf w}  to  {\bf u} ) to augment the target prey population. This strategy is essential because when the predator-prey action happens on  {\bf u}  and  {\bf v}  populations, the  {\bf u}  population will be below the threshold for growth in some periods, and lastly  {\bf v}  population will decay. 
By letting  $A(w_t) = h_tw_t$ be an augmentation function, we can write
  \begin{equation}\label{eqnA1}
  \text{{\bf Model A}} : 
    \left\{\begin{aligned}  
      u_{t+1} & = f_{\bf u}(u_t) - \delta_1 f_{\bf u}(u_t)v_t + A(f_{\bf w}w_t),\\
     v_{t+1} & = f_{\bf v}\left(v_t + \delta_2f_{\bf u}(u_t)v_t \right),\\
     w_{t+1} & = f_{\bf w}(w_t)-A(f_{\bf w}w_t).
 \end{aligned}\right.
\end{equation}
Now, substituting the growth or decay of the prey, predator and the reserve populations' functions and augmentation function into (\ref{eqnA1}) with control effort  {\bf h}, the order of event of {\bf Model A} becomes

    \begin{equation}\label{predatoreqn2}
    \left\{\begin{aligned}
     & u_{t+1}=u_t \left[s\left(1-\dfrac{u_t}{k_u}\right)\left(\dfrac{u_t}{k_u}-m\right)+1 \right] (1- \delta_1 v_t)  + h_tw_t\left[ q\left(1-\dfrac{w_t}{k_w}\right)\left(\dfrac{w_t}{k_w}-n\right)+1\right],\\ 
      &v_{t+1}= \left(v_t+ u_t \left[s\left(1-\dfrac{u_t}{k_u}\right)\left(\dfrac{u_t}{k_u}-m\right)+1 \right]
     \delta_2 v_t\right) (1- \gamma),\\
     &w_{t+1}  = (w_t-h_tw_t)\left[ q\left(1-\dfrac{w_t}{k_w}\right)\left(\dfrac{w_t}{k_w}-n\right)+1\right],\\
 \end{aligned}\right.
\end{equation}
where Equation~(\ref{vectoreqn}) represents the vector of each state variable with the initial conditions in Equation~(\ref{initialdiscrete}). Again, we assume initially that the prey (target) species has an initial population  $u_0$ below the threshold for growth and the reserve species has an initial population $w_0$ above the minimum threshold for growth and consider the constraint  $w_0 >nk_w$.  The objective functional is to maximize

\begin{equation}\label{objective1}
    J({\bf h}) = u_T + Nw_T- \sum_{t=0}^{T-1} \left(M_1h^2_t +M_2h_t\right)
\end{equation}
over 
${\bf h} \in \Omega$, the control set.  The weight $N$ ( $0<N<1$) is the constant for maximizing the reserve population at the final time,  $M_i>0$ for $i = 1,2$  are the cost constants, and $h_t$ is the proportion of the reserve population to be moved to augment the target population at each time step $(t = 0, 1, 2, \dots, T-1).$  Note that the prey/target population will be maximized at the final time, balancing the other terms in the objective functional. The  control set is given by
\begin{equation} \label{admicontrol}
    \Omega =\left\{ {\bf h} = (h_0, h_1, \dots, h_{T-1}) \hspace{3pt} | \hspace{3pt}0 \leq h_t \leq A, \hspace{10pt} t = 0, 1, \dots, T-1\right\},
\end{equation}
where $A$ is the maximum control effort. For instance, if $
A=0.7$, then at most 70\% of the reserve population can be moved for augmentation. In addition, the state variable vectors have an additional component than the control vectors.  The cost terms account for linear and nonlinear effects in the costs of translocating individuals from the reserve to the target region. We assume that the cost constants $M_1>0$ and $M_2 \geq 0$ in the rest of the work.

In the subsequent section, we will employ the extension of Pontryagin's Maximum Principle $(\mathcal{PMP})$ \cite{pontryagin1962mathematical} to obtain the necessary conditions satisfying an optimal control and the corresponding states in our discrete Equations~(\ref{predatoreqn2}) - (\ref{admicontrol}). Applying the generalization of $\mathcal{PMP}$  for optimal control problems with  discrete state systems \cite{lenhart2007optimal}, 
we form the Hamiltonian,  $\mathcal{H}_t$:

\begin{equation}\label{Heqn}
 \begin{aligned}
\mathcal{H}_t = -M_1h^2_t -M_2h_t & + \lambda_{u, t+1} \left \{u_t \left[s\left(1-\dfrac{u_t}{k_u}\right)\left(\dfrac{u_t}{k_u}-m\right)+1 \right] (1- \delta_1 v_t) \right\}\\
        & +  \lambda_{u, t+1} \left \{ h_tw_t\left[ q\left(1-\dfrac{w_t}{k_w}\right)\left(\dfrac{w_t}{k_w}-n\right)+1\right] \right\}\\ 
        & +  \lambda_{v, t+1} \left \{ \left(v_t+ u_t \left[s\left(1-\dfrac{u_t}{k_u}\right)\left(\dfrac{u_t}{k_u}-m\right)+1 \right]
     \delta_2 v_t\right) (1- \gamma) \right\}\\ 
      & + \lambda_{w, t+1} \left \{ (w_t-h_tw_t)\left[ q\left(1-\dfrac{w_t}{k_w}\right)\left(\dfrac{w_t}{k_w}-n\right)+1\right] \right \}.
    \end{aligned}
    \end{equation}
    Since the Hamiltonian (\ref{Heqn}) satisfies the concavity condition
      \begin{equation}
         \dfrac{\partial^2 \mathcal{H}_t }{\partial h^2_t} = -2M_1 \leq 0,
      \end{equation}
      for all time steps $t$  and $M_1>0$, it gives way to use the discrete version of $\mathcal{PMP}$ for {\bf Model A}, \cite{canon1970theory}.
Therefore, we derive the necessary conditions in the following theorem using the Hamiltonian (\ref{Heqn}).

\begin{Theorem}\label{apendix1}
Given an optimal control ${\bf h}^\ast \in \Omega$, $({\bf h}^\ast = (h^\ast_0, h^\ast_1, \dots, h^\ast_{T-1}))$ and the corresponding states solutions ${\bf u}^\ast =(u^\ast_0, u^\ast_1, \dots, u^\ast_T)$, ${\bf v}^\ast =(v^\ast_0, v^\ast_1, \dots, v^\ast_T)$ and ${\bf w}^\ast =(w^\ast_0, w^\ast_1, \dots, w^\ast_T)$, then from Equations~(\ref{predatoreqn2}) - (\ref{admicontrol}) there exists adjoint functions $\lambda_{{\bf u}} =(\lambda_{u,0},\lambda_{u,1}, \dots, \lambda_{u,T})$,  $\lambda_{{\bf v}} = (\lambda_{v,0},\lambda_{v,1}, \dots, \lambda_{v,T})$ and  $\lambda_{{\bf w}} =(\lambda_{w,0},\lambda_{w,1}, \dots, \lambda_{w,T})$ satisfying the adjoint equations:

 \begin{equation}\label{adjointeqn1}
    \left\{\begin{aligned}
     \lambda_{u,t} = &\bigg(\left(1 -\delta_{1} v^\ast_t\right)\lambda_{u,t+1}+ \delta_{2} v^\ast_t \left(1-\gamma \right)\lambda_{v,t+1} \bigg) \left[s \left(1-\frac{u^\ast_t}{k_{u}}\right) \left(\frac{u^\ast_t}{k_{u}}-m\right)+1\right] \\ 
     &+\bigg( u^\ast_t  \left(1-\delta_{1} v^\ast_t\right)\lambda_{u,t+1} +  u^\ast_t  \delta_{2} v^\ast_t \left(1-\gamma \right)\lambda_{v,t+1} \bigg) \left[\frac{s}{k_{u}}\left(1-\frac{u^\ast_t}{k_{u}}\right)-\frac{s}{k_{u}} \left(\frac{u^\ast_t}{k_{u}}-m\right)\right],\\ 
      \lambda_{v,t} =  & -\lambda_{u,t+1} u^\ast_t \left[s \left(1-\frac{u^\ast_t}{k_{u}}\right) \left(\frac{u^\ast_t}{k_{u}}-m\right)+1\right] \delta_{1}\\
      &+\lambda_{v,t+1} \left[1+u^\ast_t \left(s \left(1-\frac{u^\ast_t}{k_{u}}\right) \left(\frac{u^\ast_t}{k_{u}}-m\right)+1\right) \delta_{2}\right] \left(1-\gamma \right),\\
    \lambda_{w,t} = &  \bigg(h^\ast_t \lambda_{u,t+1}  + \left(1-h^\ast_t\right) \lambda_{w,t+1}\bigg)  \left[q \left(1-\frac{w^\ast_t}{k_{w}}\right) \left(\frac{w^\ast_t}{k_{w}}-n\right)+1\right]\\
     & + w^\ast_t\bigg(h^\ast_t\lambda_{u,t+1}  + \left(1-h^\ast_t \right)\lambda_{w,t+1} \bigg) \left[ \frac{q}{k_{w}} \left(1-\frac{w^\ast_t}{k_{w}}\right)- \frac{q}{k_{w}}\left(\frac{w^\ast_t}{k_{w}}-n\right)\right],
 \end{aligned}\right.
\end{equation}

with the transversality condition 

\begin{equation}
    \lambda_{u, T} = 1, \hspace{15pt}  \lambda_{v, T} = 0, \hspace{15pt}  \lambda_{w, T} = N.
\end{equation}
Moreover, the characterization of $h^\ast$ is given by
\begin{equation}\label{charact1}
      h^{\ast}_{t} = \min \left\{ A, \max \left\{0, \frac{\left(\lambda _{u,t+1}-\lambda _{w,t+1} \right) \left[q \left(1-\frac{w^\ast_t}{k_w}\right) \left(\frac{w^\ast_t}{k_w}-n\right)+1\right]w^\ast_t -M_2}{2 M_1}\right\} \right\},
\end{equation}
\end{Theorem}

	\begin{proof}
       Supposing the vector of control  ${\bf h}^\ast \in \Omega$, $({\bf h}^\ast = (h^\ast_0, h^\ast_1, \dots, h^\ast_{T-1}))$ with the corresponding state's solutions ${\bf u}^\ast =(u^\ast_0, u^\ast_1, \dots, u^\ast_T)$, ${\bf v}^\ast =(v^\ast_0, v^\ast_1, \dots, v^\ast_T)$ and ${\bf w}^\ast =(w^\ast_0, w^\ast_1, \dots, w^\ast_T)$. Using the extension of Pontryagin’s maximum principle for discrete systems \cite{lenhart2007optimal,pontryagin1962mathematical} and the Hamiltonian 
     (\ref{Heqn}), we get
     \begin{equation}\label{adjointeqn2}
    \left\{\begin{aligned}
     \lambda_{u,t} = &  \dfrac{\partial \mathcal{H}_t }{\partial u_t},\\ 
      \lambda_{v,t} =  &  \dfrac{\partial \mathcal{H}_t }{\partial v_t},\\ 
     \lambda_{w,t} = &   \dfrac{\partial \mathcal{H}_t }{\partial w_t},
 \end{aligned}\right.
\end{equation}
which gives the following results:
   \begin{equation}\label{adjointt1}
    \begin{aligned}
     \lambda_{u,t} = &\lambda_{u,t+1} \left[s \left(1-\frac{u^\ast_t}{k_{u}}\right) \left(\frac{u^\ast_t}{k_{u}}-m\right)+1\right] \left(1 -\delta_{1} v^\ast_t\right) \\ 
     &+\lambda_{u,t+1} u^\ast_t \left[\frac{s}{k_{u}}\left(1-\frac{u^\ast_t}{k_{u}}\right)-\frac{s}{k_{u}} \left(\frac{u^\ast_t}{k_{u}}-m\right)\right] \left(1-\delta_{1} v^\ast_t\right) \\
     &+\lambda_{v,t+1} \left[\left(s \left(1-\frac{u^\ast_t}{k_{u}}\right) \left(\frac{u^\ast_t}{k_{u}}-m\right)+1\right) \delta_{2} v^\ast_t\right] \left(1-\gamma \right),\\
     &+\lambda_{v,t+1} \left[u^\ast_t \left(\frac{s}{k_{u}}\left(1-\frac{u^\ast_t}{k_{u}}\right)-\frac{s}{k_{u}} \left(\frac{u^\ast_t}{k_{u}}-m\right)\right) \delta_{2} v^\ast_t\right] \left(1-\gamma \right),
 \end{aligned}
\end{equation}
and simplifying Equation~(\ref{adjointt1}) yields 
\begin{equation}\label{adjointsystem1}
    \begin{aligned}
     \lambda_{u,t} = &\bigg(\left(1 -\delta_{1} v^\ast_t\right)\lambda_{u,t+1}+ \delta_{2} v^\ast_t \left(1-\gamma \right)\lambda_{v,t+1} \bigg) \left[s \left(1-\frac{u^\ast_t}{k_{u}}\right) \left(\frac{u^\ast_t}{k_{u}}-m\right)+1\right] \\ 
     &+\bigg( u^\ast_t  \left(1-\delta_{1} v^\ast_t\right)\lambda_{u,t+1} +  u^\ast_t  \delta_{2} v^\ast_t \left(1-\gamma \right)\lambda_{v,t+1} \bigg) \left[\frac{s}{k_{u}}\left(1-\frac{u^\ast_t}{k_{u}}\right)-\frac{s}{k_{u}} \left(\frac{u^\ast_t}{k_{u}}-m\right)\right].\\
 \end{aligned}
\end{equation}
   Again, from the second equation in (\ref{adjointeqn2}), we have
   \begin{equation}\label{adjointsystem2}
    \begin{aligned}
      \lambda_{v,t} =  & -\lambda_{u,t+1} u^\ast_t \left[s \left(1-\frac{u^\ast_t}{k_{u}}\right) \left(\frac{u^\ast_t}{k_{u}}-m\right)+1\right] \delta_{1}\\
      &+\lambda_{v,t+1} \left[1+u^\ast_t \left(s \left(1-\frac{u^\ast_t}{k_{u}}\right) \left(\frac{u^\ast_t}{k_{u}}-m\right)+1\right) \delta_{2}\right] \left(1-\gamma \right).
 \end{aligned}
\end{equation}
Moreover, the third equation in (\ref{adjointeqn2}) gives
\begin{equation}\label{adjoint2}
    \begin{aligned}
     \lambda_{w,t} = &  \lambda_{u,t+1} h^\ast_t \left[q \left(1-\frac{w^\ast_t}{k_{w}}\right) \left(\frac{w^\ast_t}{k_{w}}-n\right)+1\right]\\
     &+\lambda_{u,t+1} h^\ast_tw^\ast_t \left[ \frac{q}{k_{w}} \left(1-\frac{w^\ast_t}{k_{w}}\right)- \frac{q}{k_{w}}\left(\frac{w^\ast_t}{k_{w}}-n\right)\right]\\
     &+ \lambda_{w,t+1}  \left(1-h^\ast_t\right) \left[q \left(1-\frac{w^\ast_t}{k_{w}}\right) \left(\frac{w^\ast_t}{k_{w}}-n\right)+1\right]\\
     &+\lambda_{w,t+1} \left(w^\ast_t-h^\ast_t w^\ast_t\right)  \left[\frac{q}{k_{w}}\left(1-\frac{w^\ast_t}{k_{w}}\right)- \frac{q}{k_{w}}\left(\frac{w^\ast_t}{k_{w}}-n\right)\right],
 \end{aligned}
\end{equation}
and simplifying  Equation~(\ref{adjoint2}), we obtain
  \begin{equation}\label{adjointsystem3}
    \begin{aligned}
     \lambda_{w,t} = &  \bigg(h^\ast_t \lambda_{u,t+1}  + \left(1-h^\ast_t\right) \lambda_{w,t+1}\bigg)  \left[q \left(1-\frac{w^\ast_t}{k_{w}}\right) \left(\frac{w^\ast_t}{k_{w}}-n\right)+1\right]\\
     & + w^\ast_t\bigg(h^\ast_t\lambda_{u,t+1}  + \left(1-h^\ast_t \right)\lambda_{w,t+1} \bigg) \left[ \frac{q}{k_{w}} \left(1-\frac{w^\ast_t}{k_{w}}\right)- \frac{q}{k_{w}}\left(\frac{w^\ast_t}{k_{w}}-n\right)\right].\\
 \end{aligned}
\end{equation}
Equations~(\ref{adjointsystem1}), (\ref{adjointsystem2}) and (\ref{adjointsystem3}) represent the  adjoint discrete equations with the transversality conditions $\lambda_{u, T} = 1$,  $\lambda_{v, T} = 0,$  $\lambda_{w, T} = N.$ \\
   In addition, the Hamiltonian differentiated with respect to the control on the interior of the control set gives 
    \begin{equation}\label{heqn2}
    \begin{aligned}
     0 =  & -2 M_{1} h_t-M_{2}+ w_t\lambda_{u,t+1} \left[q \left(1-\frac{w_t}{k_{w}}\right) \left(\frac{w_t}{k_{w}}-n\right)+1\right] \\
     &- w_t\lambda_{w,t+1} \left[q \left(1-\frac{w_t}{k_{w}}\right) \left(\frac{w_t}{k_{w}}-n\right)+1\right]\hspace{5pt} \text{at} \hspace{5pt} h^\ast_t.
 \end{aligned}
\end{equation}
Solving for $h^\ast_t$ in Equation~(\ref{heqn2}) on the interior of the control set and simplifying the terms, we get 
 \begin{equation}h^\ast_t = \frac{ \left(\lambda _{u,t+1}-\lambda _{w,t+1} \right)\left[q \left(1-\frac{w^\ast_t}{k_w}\right) \left(\frac{w^\ast_t}{k_w}-n\right)+1\right]w^\ast_t -M_2}{2 M_1}, 
 \end{equation}
which then, taking bounds into account, gives the required optimal control characterization as 
\begin{equation}
    h^{\ast}_{t} = \min \left\{ A, \max \left\{0, \frac{ \left(\lambda _{u,t+1}-\lambda _{w,t+1} \right) \left[q \left(1-\frac{w^\ast_t}{k_w}\right) \left(\frac{w^\ast_t}{k_w}-n\right)+1\right]w^\ast_t -M_2}{2 M_1}\right\} \right\}.
\end{equation}
 The state system~(\ref{predatoreqn2}), the adjoint system (\ref{adjointeqn1}), and the optimal control characterization (\ref{charact1}) give the optimality system for the {\bf Model A} of the augmentation.
   \end{proof}

Using the discrete version of the forward-backward sweep method, we numerically solve  {\bf Model A}  under distinct parameter values. This can be achieved as follows:

\begin{enumerate}
    \item [I.] First, an initial guess is made for the vector of optimal controls $({\bf h}^\ast = (h_0, h_1, \dots, h_{T-1}))$.
    \item [II.] The second step is to solve forward for the vectors $({\bf u} = \left(u_0, u_1, \dots,u_T\right))$,  $({\bf v} = \left(v_0, v_1, \dots,v_T\right))$, and $( ({\bf w} = \left(w_0, w_1, \dots, w_T\right))$ using the initial conditions $u_0, v_0$ and $ w_0$ and the current vector for  {\bf h} .
    \item[III.] Further, using the transversality conditions $\lambda_{u, T} = 1$,  $\lambda_{v, T} = 0,$  $\lambda_{w, T} = N,$ and the vectors $u,v,w$ and  {\bf h}  in step II, we solve backwards for the adjoint variables $(\lambda_{{\bf u}} =(\lambda_{u,0},\lambda_{u,1}, \dots, \lambda_{u,T}))$,  $(\lambda_{{\bf v}} = (\lambda_{v,0},\lambda_{v,1}, \dots, \lambda_{v,T}))$ and  $(\lambda_{{\bf w}} =(\lambda_{w,0},\lambda_{w,1}, \dots, \lambda_{w,T}))$.
    \item[IV.] At this stage, the control is updated using a convex combination of the previous control and the characterization using the current adjoint and state variables from the current iteration.
    \item[V.] Finally, the updated control is used in the iterated process until the successive iterates of the control values are sufficiently close \cite{kelley1999iterative, lenhart2007optimal}.
\end{enumerate}

 \subsection{ {\bf Model B}: Augmentation followed by growth and then predator-prey action} 
 The reserve population is harvested to augment the target population. 
 After the first action, augmentation from  {\bf w}  to  {\bf u},  
 we substitute the results $u_t+h_tw_t$ and $w_t-h_tw_t$ into the respective growth functions of  {\bf u}  and  {\bf w}  in Equation~(\ref{predatoreqn1}). Lastly, the predator-prey action happens.  
If $A(w_t) = h_tw_t$ is the  augmentation function, then we can have 
\begin{equation}
 \text{{\bf Model B}} : 
    \left\{\begin{aligned}  
      u_{t+1} & = f_{\bf u}(u_t + A(w_t)) - \delta_1 f_{\bf u}(u_t + A(w_t))v_t,\\
     v_{t+1} & = f_{\bf v}\left(v_t + \delta_2f_{\bf u}(u_t + A(w_t))v_t \right),\\
     w_{t+1} & = f_{\bf w}(w_t - A(w_t)),
 \end{aligned}\right.
\end{equation}
so that the state equations of {\bf Model B} can be constructed as
\begin{equation}\label{predatoreqn3}
    \left\{\begin{aligned}
      &u_{t+1} = \left(u_t+ h_tw_t\right) \left[s\left(1-\dfrac{u_t+h_tw_t}{k_u}\right)\left(\dfrac{u_t+h_tw_t}{k_u}-m\right) 
+1 \right](1-\delta_1 v_t),\\ 
    &v_{t+1} = \left(v_t +\left(u_t+ h_tw_t\right) \left[s\left(1-\dfrac{u_t+h_tw_t}{k_u}\right)\left(\dfrac{u_t+h_tw_t}{k_u}-m\right) +1 \right]  \delta_2v_t \right) \left( 1- \gamma  \right),\\ 
      & w_{t+1}  = \left(w_t- h_tw_t\right)\left[ q\left(1-\dfrac{w_t-h_tw_t}{k_w}\right)\left(\dfrac{w_t-h_tw_t}{k_w}-n\right)+1\right],
 \end{aligned}\right.
\end{equation}

with the same initial conditions  (\ref{initialdiscrete}) and $t = 0, 1, \dots, T-1$. The same assumptions that the prey (target) species has an initial population $u_0$ below the threshold for growth and the reserve species has an initial population $w_0$ above the threshold for growth and other assumptions applied in the first scenario are valid here. Again, the objective functional $J({\bf h})$ in Equations~(\ref{objective1})-(\ref{admicontrol}) remains the same in this strategy. The Hamiltonian function for {\bf Model B} is then constructed as
 \begin{align}\notag
\mathcal{H}_t = &-M_1h^2_t -M_2h_t \\ \notag
& + \lambda_{u, t+1} \left \{ \left(u_t+ h_tw_t\right) \left[s\left(1-\dfrac{u_t+h_tw_t}{k_u}\right)\left(\dfrac{u_t+h_tw_t}{k_u}-m\right) 
+1 \right](1-\delta_1 v_t)\right\}\\ \notag
        & +  \lambda_{v, t+1} \left \{ \left(v_t +\left(u_t+ h_tw_t\right) \left[s\left(1-\dfrac{u_t+h_tw_t}{k_u}\right)\left(\dfrac{u_t+h_tw_t}{k_u}-m\right) +1 \right]  \delta_2v_t \right) \left( 1- \gamma  \right) \right\}\\  \label{heqn3}
      & + \lambda_{w, t+1} \left \{ \left(w_t- h_tw_t\right)\left[ q\left(1-\dfrac{w_t-h_tw_t}{k_w}\right)\left(\dfrac{w_t-h_tw_t}{k_w}-n\right)+1\right] \right \}.
    \end{align}

    Next, we simplify the Hamiltonian function (\ref{heqn3}) and obtain
\begin{align}\notag
\mathcal{H}_t = &-M_1h^2_t -M_2h_t  + \lambda_{u, t+1} \left \{ \big[ \dfrac{1}{k^2_u}\left( (1+m)sk_uu_t-(ms-1)k^2_u-su_t \right)u_t \right.\\ \notag
 & +\dfrac{1}{k^2_u}\left(2(1+m)sk_uu_t-3su^2_t+k^2_u(1-ms)\right) w_t h_t - \dfrac{s}{k^2_u}\left(3u_t-k_u(1+m) \right)w^2_th^2_t \\ \notag
 &\left.- \dfrac{s}{k^2_u}w^3_th^3_t \big] (1-\delta_1 v_t)\right\} +  \lambda_{v, t+1} \left\{  \big[ \big(1+\dfrac{1}{k^2_u}\left( (1+m)sk_uu_t-(ms-1)k^2_u-su_t \right)\delta_2u_t\big)v_t \right.\\  \notag
 & +\dfrac{1}{k^2_u}\left(2(1+m)sk_uu_t-3su^2_t+k^2_u(1-ms)\right) \delta_2v_tw_t h_t  -\dfrac{s}{k^2_u}\left(3u_t-k_u(1+m) \right)\delta_2v_tw^2_th^2_t\\ \notag
 &\left.- \dfrac{s}{k^2_u}\delta_2v_tw^3_th^3_t \big] (1-\gamma)\right\} + \lambda_{w, t+1} \left \{ \frac{1}{k^2_w}\big[ q(1+n)k_ww_t-qw^2_t+(1-nq)k^2_w)w_t \right. \\  \notag
 & + \left.\dfrac{1}{k^2_w} \left(3qw^2_t-2qk_w(1+n)w_t+(qn-1)\right)w_th_t - \dfrac{1}{k^2_w}\left( q(w_t+2)-(q-n)k_w \right)w^2_th^2_t +\dfrac{q}{k^2_w}w^3_t h^3_t \right\}.\\ \label{heqn4}
 \end{align}
  The Hamiltonian $\mathcal{H}_t$ in Equation~(\ref{heqn4}) is a cubic expression in the control $h_t$. The concavity of $\mathcal{H}_t$ is difficult to calculate, and finding the optimality condition is not feasible due to the complicated dependencies on the control. Thus, we turn to an alternative optimization method.

We therefore employ a direct optimization technique that maximizes the objective functional J({\bf h})  without the adjoint equations. For this reason, we resort to the sequential quadratic programming (SQP) iterative method, one of the most effective methods for nonlinearly constrained optimization, which generates steps by solving quadratic subproblems \cite{nocedal1999numerical,paul1995SQP}. The  SQP  (also called the Lagrange--Newton method \cite{fletcher2000practical}) transforms a nonlinear optimization problem into a sequence of quadratic optimization problems (quadratic objective, linear equality, and inequality constraints), which are simpler to solve \cite{bonnans2006numerical}. The basic idea of SQP is to model a problem such as (\ref{objective1}) and (\ref{predatoreqn3}) at a given approximate solution,  say $u_t, v_t, w_t$, by a quadratic programming subproblem, and then to use the solution to this subproblem to construct a better approximation $u_{t+1}, v_{t+1}, w_{t+1}$. We assume the concavity condition in our maximization problem defined in {\bf Model B} is not specified. Also, we assume that all of the functions in (\ref{objective1}) and (\ref{predatoreqn3}) are twice continuously differentiable. For a detailed explanation of this method, see  \cite{bonnans2006numerical,paul1995SQP}. The method is outlined as follows:

 \begin{enumerate}
    \item [I.]  An initial guess/ iterate is made for the state vectors $(({\bf u} = \left(u_0, u_1, \dots,u_T\right))$,  $(({\bf v} = \left(v_0, v_1, \dots,v_T\right))$, $(({\bf w} = \left(w_0, w_1, \dots, w_T\right))$ and the vector of optimal controls $(({\bf h} = (h_0, h_1, \dots, h_{T-1}))$ using the initial conditions $u_0, v_0$ and $ w_0$ and the bounds on the controls
    \item [II.]  The initial gradient of the objective functional $J({\bf h})$ is computed which will be used in the construction and solving of the quadratic subproblem
    \item [III.]  The next step is the formulation of the quadratic subproblem obtained by the approximation of the objective functional $J({\bf h})$ and linearization of the state constraints
    \item[IV.] Using the MATLAB function $fmincon$ solver with $SQP$ algorithm, the quadratic subproblem is solved
    \item[VI.] Next, the control variable $h_{t+1}$ is updated and the state variables $u_{t+1}, v_{t+1}$ and $w_{t+1}$ are updated using the newly  updated control
     \item[VII.] To check for convergence, new gradients of the states and the objective functional are computed. Then the algorithm stops if the Karush-Kuhn-Tucker (KKT) condition for optimization holds. Otherwise, increase the time step $t$ by $1$ and repeat the steps. 
\end{enumerate}

 \begin{Remark}
  The primary reason for using a quadratic subproblem for the optimization in  {\bf Model B} (with a quadratic objective functional and state constraints) is that quadratic subproblems are relatively easy to solve and yet, in their objective functional, can reflect the nonlinearities of the original problem. This quadratic subproblem is generated by carrying out a second-order Taylor expansion of the objective functional around the current iterate. The gradient determines the linear term, whereas the Hessian gives the quadratic term. The constraints can also be linearized using first-order Taylor expansions. The  MATLAB function $fmincon$ solver is used to solve this quadratic subproblem. 
\end{Remark} 

\section{Numerical results and discussion}
\label{sec3}
  The optimality systems of the two orders of events will be solved by iterative methods using MATLAB software. Thus, forward-solving of the state system followed by backward-solving of the adjoint system of {\bf Model A}  and employing sequential quadratic programming that uses the fmincon function with $SQP$ algorithm to solve  {\bf Model B}. 
 In the simulations of both models, we set the length of time corresponding to each breeding season as $T = 6$, and the maximum proportion of the reserve species to be translocated at any particular time step to be 70\%, i.e., $A = 0.70$. 
 we set 
 $m=n=0.25$ 
 and consider the constraint $w_0 >n k_w$, that is, initially the size of the target population is below the Allee threshold, and the reserve population size is above the threshold condition, which allows for harvesting. This work does not study a specific species; however, it provides a general augmentation model for any predator-prey relationship, hence all the parameter values and initial conditions used in this research are hypothetical. We assume the population units are 1000 individuals. Throughout the simulations, the initial conditions and the carrying capacities are expressed as products of the units of the population (1000 individuals). We set the initial conditions $ {\bf u}_0 = 0.20;  {\bf v}_0 = 0.5; $  $ {\bf w}_0 = 0.70$ with the following set of baseline parameters 
\begin{equation}\label{baselineeqn}
    \begin{aligned}
        T = & 6, s=0.25, k_u = 0.50, m = 0.25, \delta_1 = 0.40, \delta_2 = 0.50, \gamma = 0.025,\\
        &q = 0.85, k_w = 0.80, n = 0.25, M_1 = 0.40, M_2 = 0.15, N = 0.50,
    \end{aligned}
\end{equation} 

 for the simulations of {\bf Model A} and {\bf Model B}. To  examine different dynamical behaviors of the models, we vary the parameters $\gamma, q, M_1, M_2,$ and $N$ while the rest are chosen as in Equation~(\ref{baselineeqn}). To further assess the robustness of our model and the effectiveness of our numerical solutions, we present values of the objective functional with no optimal control and values of the objective functional with optimal control.  In the simulations, the plots of the prey populations with optimal controls are indicated in black, the plots of the reserve population with optimal controls are indicated in blue, and the plots of the predator populations with optimal controls are indicated in red. Again, gray dotted lines indicating the Allee thresholds ($m = n = 0.25$) in the prey and the predator populations are shown in all the plots.  Figures~\ref{Figure1A} and \ref{Figure1B} show additional plots without augmentation, indicated in red dash--dot lines.  The simulation results are presented in the subsequent subsections with discussions.

    \subsection{Simulation of the models (with no control and with optimal control)}
  The simulations in this subsection provide the baseline parameter values in Equation~(\ref{baselineeqn}). Using the discrete version of the forward-backward sweep method outlined in the latter part of Section~\ref{sec 4.1} and the baseline parameter values, we present the plots of {\bf Model A} in
 Figure~\ref{Figure1A}. Also, using the SQP iterative method with baseline parameters, the plots of {\bf Model B} are presented in Figure~\ref{Figure1B}.

	\begin{figure} [H]
			\centering
    \includegraphics[width=1.1\linewidth]{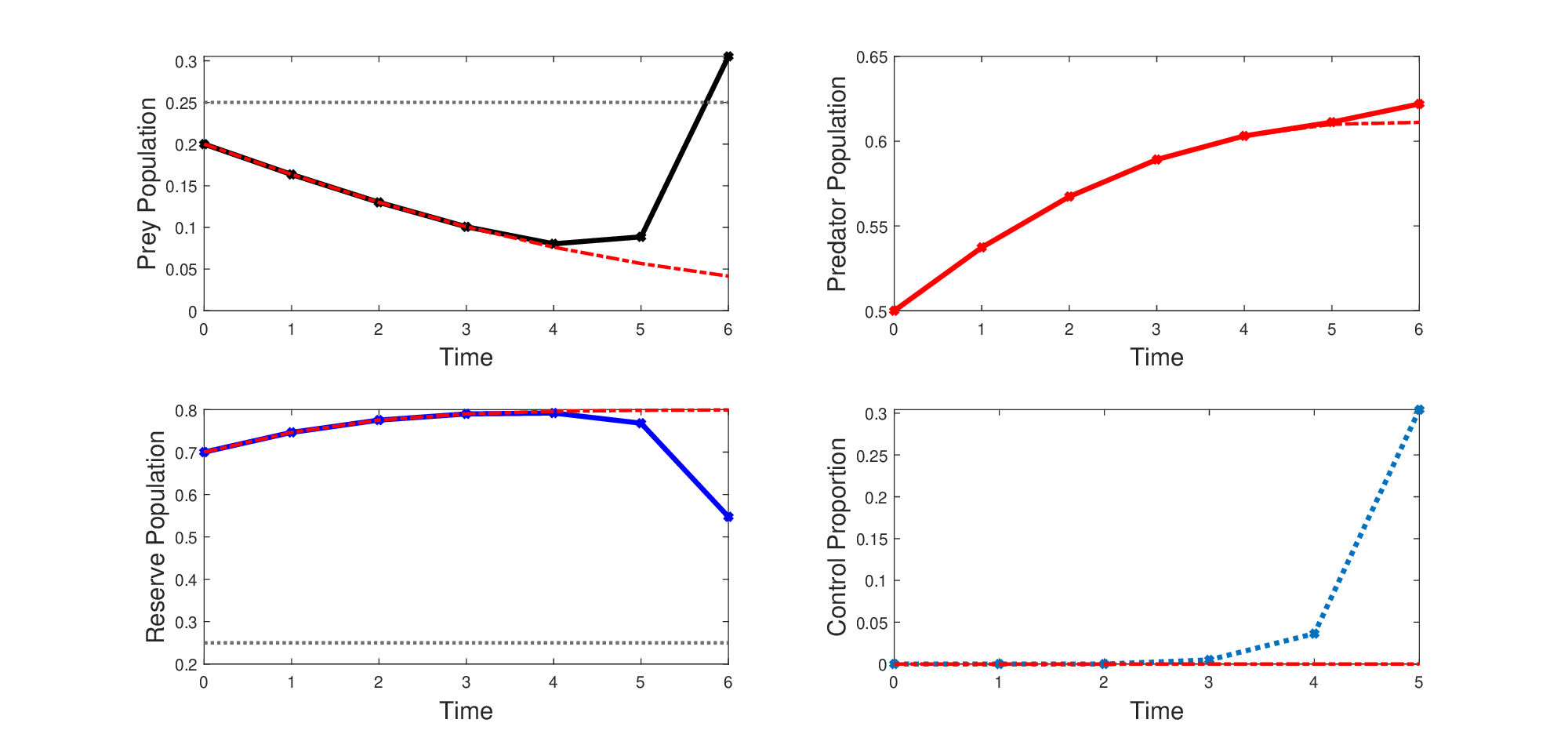}
		\caption{Plots of the states with no control and plot with optimal control of the discrete augmentation {\bf Model A} where the population is allowed to grow and then predator-prey action happens before augmentation at each time step using the baseline parameter values (\ref{baselineeqn}). The red dash-dot lines indicate the plots of the populations without augmentation, and the gray dotted lines indicate the Allee thresholds ($m = n = 0.25$) for the prey and the reserve populations, respectively. The corresponding optimal control and the objective functional values are: ${\bf h}^\ast = [0,0,0,0,0.04,0.30]$,  $J(0) = 0.4413$ (with no control), and $J({\bf h}^{\ast}) = 0.4896$ (with optimal control).} \label{Figure1A}
	\end{figure}

  The optimal control values in Figure~\ref{Figure1A}, ${\bf h}^\ast = [0,0,0,0,0.04,0.30]$, indicate that at each time step $t = 0,1,2,$ and $3$, no individuals from the reserve population are translocated to the target population. However, at time steps $t=4$ and $5$, about $4\%$ and $30\%$ of the reserve population are translocated to augment the target population, respectively. This increases the target population size to $0.30$ above the minimum threshold for growth by $0.05$.  In the same manner, in Figure~\ref{Figure1B},  the optimal control values,  ${\bf h} = [0,0, 0,0.05,0.08,0.15]$,  indicate that augmentation occurs at time steps $t =3, 4,5$, where about $5\%$, $8\%$ and $15\%$ of the reserve population can be translocated to augment the target population, respectively. In this case, the target population showed a rise in the population level when augmentation began; however, the population size of the target species is still below the minimum threshold for growth at the final time using the same baseline values for the simulation. Therefore, the results of the optimal strategy grow and then predator-prey action and then augment in Figure~\ref{Figure1A}  indicate a higher level of the target population size than the optimal strategy augment and then grow and then predator-prey action obtained in Figure~\ref{Figure1B} at the end of the final time steps ($T = 6$) using the same baseline parameters.  This result shows the importance of employing different orders of events in the discrete models.

	\begin{figure} [H]
			\centering
		\includegraphics[width=1.1\linewidth]{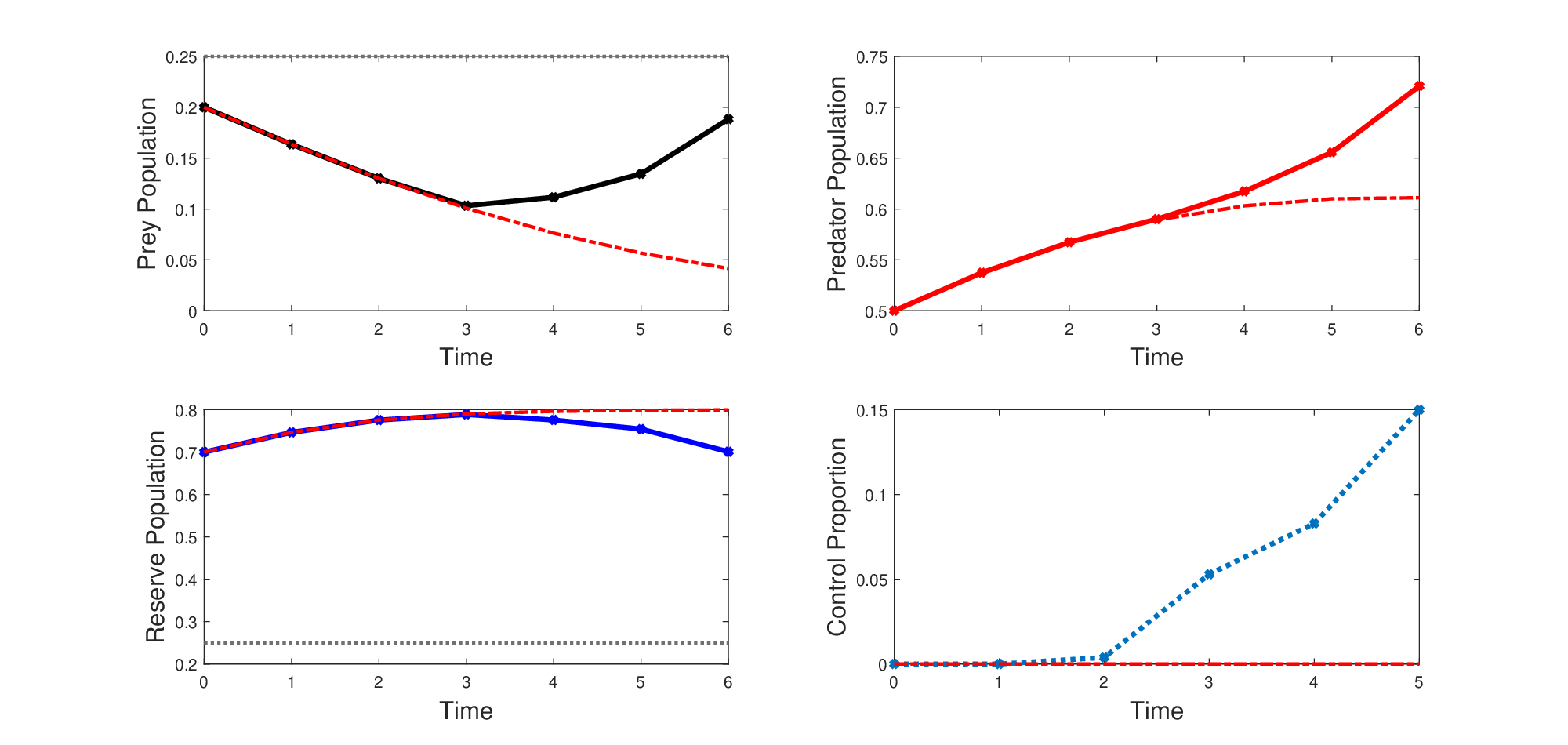}
		\caption{Plots of the states with no control and plot with optimal control of the discrete augmentation {\bf Model B} where the population is augmented and then grows and then predator-prey action at each time step using the baseline parameter values (\ref{baselineeqn}). The red dash-dot lines indicate the plots of the populations without augmentation, and the gray dotted lines indicate the Allee thresholds ($m = n = 0.25$) for the prey and the reserve populations, respectively. The corresponding optimal control and the objective functional values are:   ${\bf h}^\ast = [0,0, 0,0.05,0.08,0.15]$, $J(0) = 0.4413$ (with no control), $J({\bf h}^{\ast}) = 0.4825$ (with optimal control).} \label{Figure1B}
	\end{figure}

       The predator population levels in Figures~\ref{Figure1A} and~\ref{Figure1B}  show a rise at the end of the final time step when optimal controls are implemented.  Note that in this work, the target population is of interest; however, maximizing the target population provides more resources (prey) for the predators, which further results in an increase in the population levels of the predator population. Again, with regards to the objective functional values using the baseline plots, Figure~\ref{Figure1A}  gives $J(0) = 0.4413$ (with no control) and $J({\bf h}^{\ast}) = 0.4896$ (with optimal control), which means that implementing the optimal control in {\bf Model A}  gives  $11\%$ higher objective functional value than when there is no optimal control. Similarly, Figure~\ref{Figure1B} gives  $J(0) = 0.4413$ (with no control) and $J({\bf h}^{\ast}) = 0.4825$ (with optimal control), indicating $9\%$ higher objective functional value of implementing the optimal augmentation {\bf Model B} than when there is no control. The novel results obtained indicate the importance of considering different orders of events. In addition, the reserve population levels in Figures~\ref{Figure1A} and~\ref{Figure1B} decline when the optimal control values are not zero due to the movement of some of the individuals.

    The optimal control values at the time step $t = 0,1, 2,3$  in Figure~\ref{Figure1A} and $t = 0,1, 2$ in Figure~\ref{Figure1B} are all zero,  indicating no individual is translocated for augmentation at these time steps and this applies to all the time steps where the optimal control values are zero. The optimal control values are $30\%$ for {\bf Model A}  and $15\%$ for {\bf Model B} at the final time steps in Figures~\ref{Figure1A} and~\ref{Figure1B}, respectively. The higher the optimal control value is, the more individuals are translocated from the reserve population to the target region. 
     The predator population levels at the final time with augmentation are approximately $0.62$ and $0.72$ in Figures~\ref{Figure1A} and~\ref{Figure1B}, respectively, as compared to the previously population levels of approximately $0.61$ and $0.56$, respectively.

   \subsection{Effect of  weight and cost constants}
   This subsection provides more insights into how the weight and cost constants influence our models.
   Figures~\ref{Figure2A} and~\ref{Figure2B} show results of setting only $M_2=0$ and keeping the baseline parameter values constant. Also, Figures~\ref{Figure3A} and~\ref{Figure3B}  show plots of the models with $N = 0.10$ and $M_2=0$.

   \begin{figure} [H]
			\centering
			\includegraphics[width=1.1\linewidth]{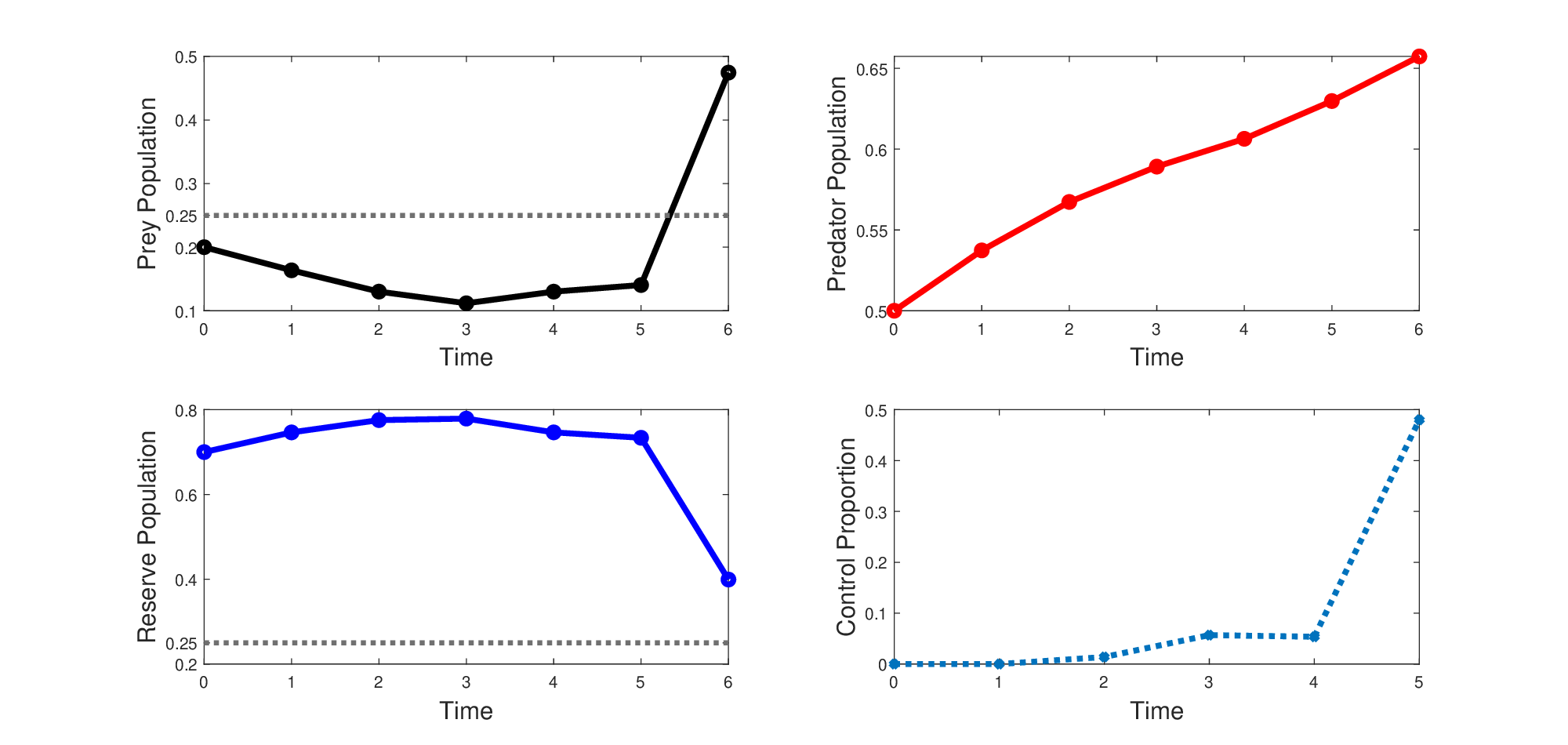}
		\caption{ Plots of the states and plot of the optimal control of the discrete augmentation {\bf Model A} where the population is allowed to grow and then predator-prey action happens before augmentation at each time step using the baseline parameter values (\ref{baselineeqn}) except for $ M_2 = 0$. The corresponding optimal control and the objective functional values are: ${\bf h}^\ast = [0,0,0.01,0.06,0.05,0.48]$, $J(0) = 0.4413$ (with no control), $J({\bf h}^{\ast}) = 0.5794$ (with optimal control).} \label{Figure2A}
	\end{figure}

   \begin{figure} 
			\centering
			\includegraphics[width=1.1\linewidth]{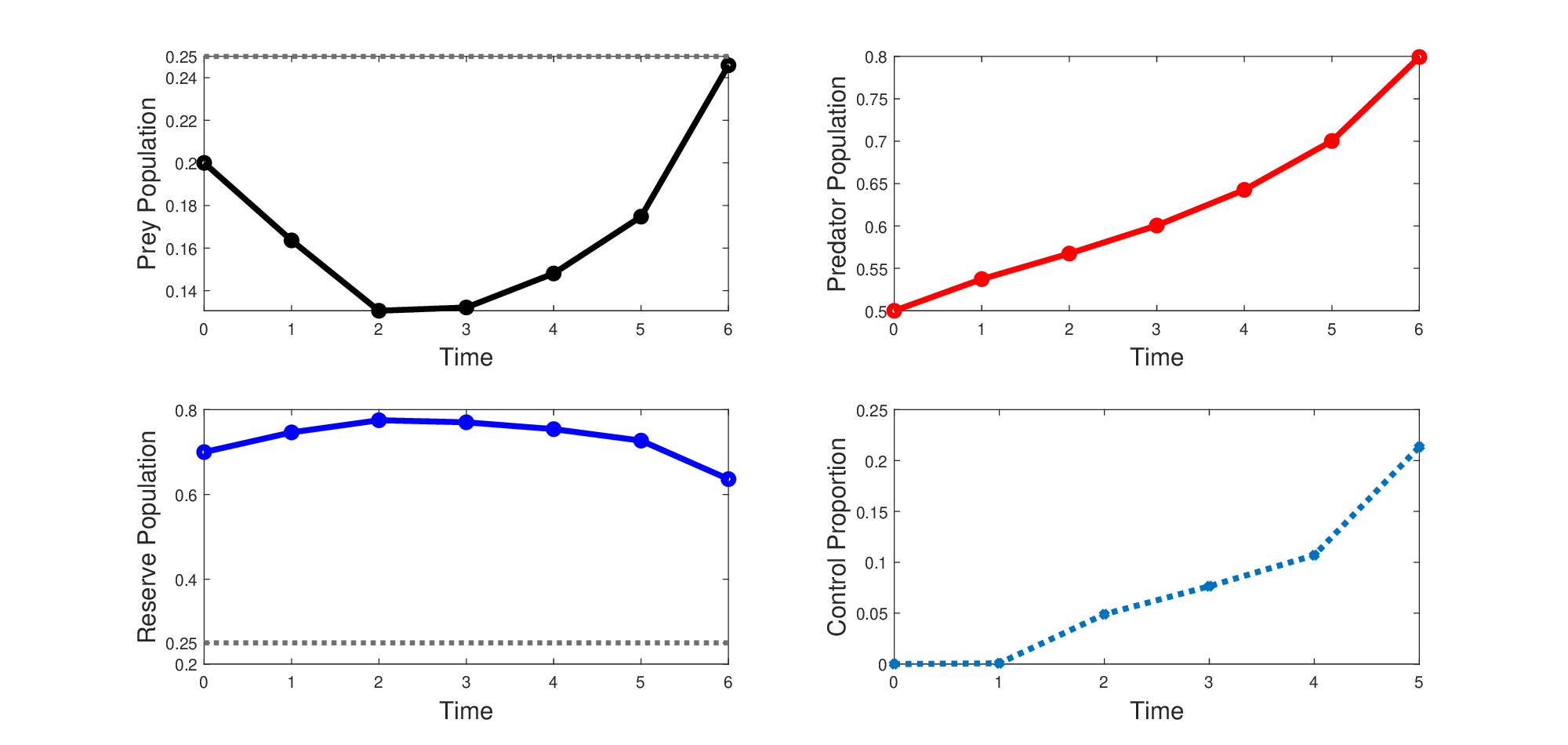}
		\caption{Plots of the states and plot of the optimal control of the discrete augmentation {\bf Model B} where the population is augmented and then grows and then predator-prey action at each time step using the baseline parameter values (\ref{baselineeqn}) except for $ M_2 = 0$. The corresponding optimal control and the objective functional values are:  ${\bf h}^\ast = [0,0,0.05,0.08,0.11,0.21]$, $J(0) = 0.4413$ (with no control), $J({\bf h}^{\ast})= 0.5379$ (with optimal control).} \label{Figure2B}
	\end{figure}

     In  Figure~\ref{Figure2A}, the optimal control values, ${\bf h}^\ast = [0,0,0.01,0.06,0.05,0.48]$ show that at each time step $t = 2, 3, 4$ and $5$, about $1\%, 6\%, 5\%$ and $48\%$ individuals are translocated respectively from the reserve population to the target region for augmentation when the order of event is growth followed by predator-prey action and then augmentation. Similarly, the optimal control values,  ${\bf h}^\ast = [0,0,0.05,0.08,0.11,0.21]$, in Figure~\ref{Figure2B} indicate that at each time steps $t = 2, 3, 4$ and $5$, about $5\%, 8\%, 11\%$ and $21\%$ are translocated respectively from the reserve population to the target region for augmentation when the order of event is to augment followed by growth and then predator-prey action. The objective functional values  $J(0) = 0.4413$ (with no control), $J({\bf h}^{\ast}) = 0.5794$ (with optimal control) and  $J(0) = 0.4413$ (with no control), $J({\bf h}^{\ast})= 0.5379$ (with optimal control) obtained in  Figures~\ref{Figure2A} and~\ref{Figure2B} respectively gives $31\%$ and $22\%$ higher objective functional values for the models with optimal control than when there is no control action.  
   
  \begin{figure} 
			\centering
			\includegraphics[width=1.1\linewidth]{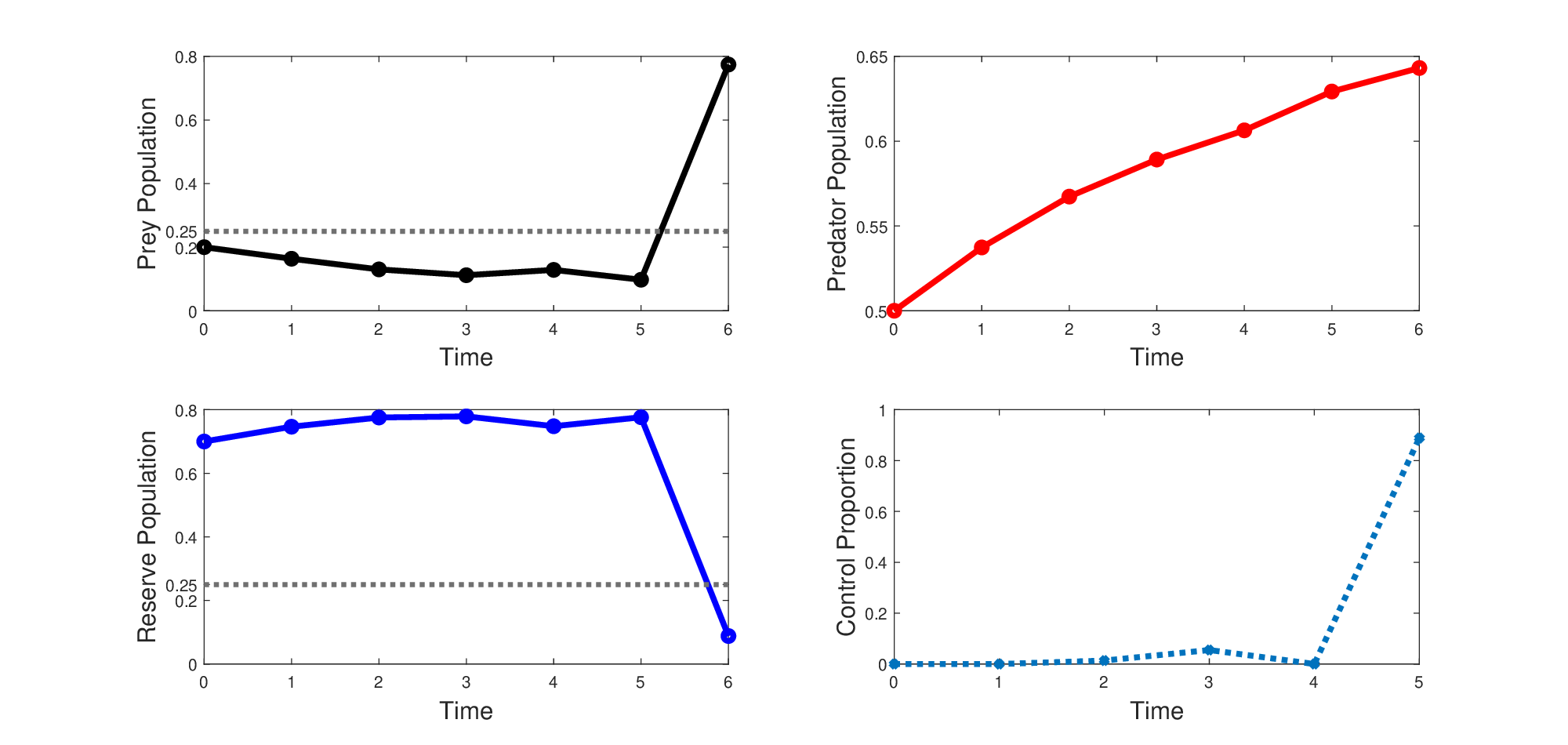}
		\caption{Plots of the states and plot of the optimal control of the discrete augmentation {\bf Model A} where the population is allowed to grow and then predator-prey action happens before augmentation at each time step using the baseline parameter values (\ref{baselineeqn}) except for $M_2 = 0, N= 0.10$. The corresponding optimal control and the objective functional values are: $ {\bf h}^\ast = [0, 0, 0.01, 0.06, 0, 0.9]$, $J(0) = 0.1215$ (with no control), $J({\bf h}^{\ast}) = 0.4662$ (with optimal control).} \label{Figure3A}
	\end{figure}

\begin{figure} [H]
			\centering
			\includegraphics[width=1.1\linewidth]{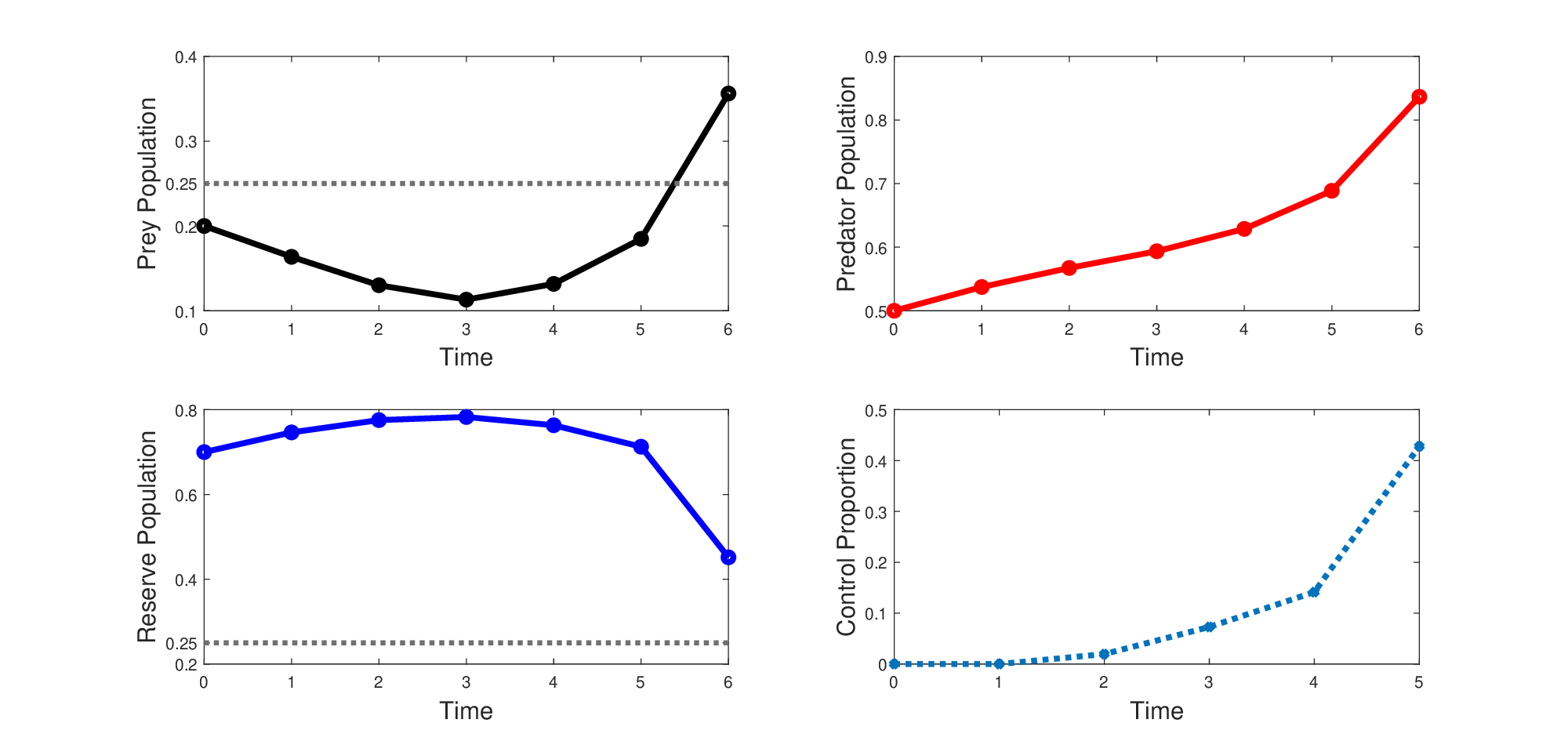}
		\caption{Plots of the states and plot of the optimal control of the discrete augmentation {\bf Model B} where the population is augmented and then grows and then predator-prey action at each time step using the baseline parameter values (\ref{baselineeqn}) except for $ M_2 = 0, N= 0.10$. The corresponding optimal control and the objective functional values are:  $ {\bf h}^\ast = [0, 0, 0.02,0.07,0,0.14]$, $J(0) = 0.1215$ (with no control), $J({\bf h}^{\ast}) = 0.3178$ (with optimal control).} \label{Figure3B}
	\end{figure}

     In  Figures~\ref{Figure3A} and~\ref{Figure3B}, we set $M_2 = 0$ and $N=0.10$ (reducing the weight constant for the reserve population from $0.50$ to $0.10$) and plots of the states and plot of the optimal control of {\bf Model A}  and {\bf Model B} are shown in  Figures~\ref{Figure3A} and~\ref{Figure3B}, respectively. The optimal control values, $ {\bf h}^\ast = [0, 0, 0.01, 0.06, 0, 0.9]$, show that about $1\%$, $6\%$ and $90\%$ of the reserve population at times $t = 2, 3$, and $5$ respectively are translocated to augment the target population with the event in {\bf Model A} that yields a high target population size of 0.8 at the end of the period. Likewise, the optimal control values,   $ {\bf h}^\ast = [0, 0, 0.02,0.07,0,0.14]$, show that about $2\%$, $7\%$ and $14\%$ of the reserve population at times $t = 2, 3$, and $5$ respectively are translocated to augment the target population with the event in {\bf Model B} causing a rise in the target population level to $0.37$ which is above the minimum threshold for growth.

   The objective functional values   $J(0) = 0.1215$ (with no control), $J({\bf h}^{\ast}) = 0.4662$ (with optimal control) and  $J(0) = 0.1215$ (with no control), $J({\bf h}^{\ast}) = 0.3178$ (with optimal control) from Figures~\ref{Figure3A} and~\ref{Figure3B} respectively gives $284\%$ and $162\%$ higher objective functional values for adopting the optimal strategies than when there is no
 control.  Moreover, for $M_2 =0$, {\bf Model A} gives  $48\%$ of control values at the final time and $90\%$ for $M_2 =0, N=0.10$. Therefore, when $M_2=0$ and $N$ is reduced from $0.50$ to $0.10$ with the rest of the parameters fixed in (\ref{baselineeqn}), more of the reserve species are translocated to the target region, causing the target population level to rise to $0.79$.  Hence, the optimal strategies (setting $M_2 = 0$ and $N=0.10$) in Figures~\ref{Figure3A} and~\ref{Figure3B} give different objective functional values with optimal control and the target population levels, which are different from the results in the Figures~\ref{Figure1A} and~\ref{Figure1B}, respectively, and also in Figures~\ref{Figure2A} and~\ref{Figure2B}, respectively.

     \subsection{Effect of the reserve population's growth rate and carrying capacity }
     The impact of reducing the growth rate and carrying capacity of the reserve species on the different optimal orders of events is investigated in this section. The intrinsic growth rate of the reserve species, $q$, is reduced from $0.85$ to $0.70$, and the environmental carrying capacity of the reserve population, $k_w$, is also reduced from $0.80$ to $0.60$, with all the initial conditions and the remaining baseline parameter values (\ref{baselineeqn}) remaining the same.  Figure~\ref{Figure4A} shows the time series plot of {\bf Model A} whilst Figure~\ref{Figure4B} shows plots of {\bf Model B} under this scenario.

    \begin{figure} [H]
			\centering
			\includegraphics[width=1.1\linewidth]{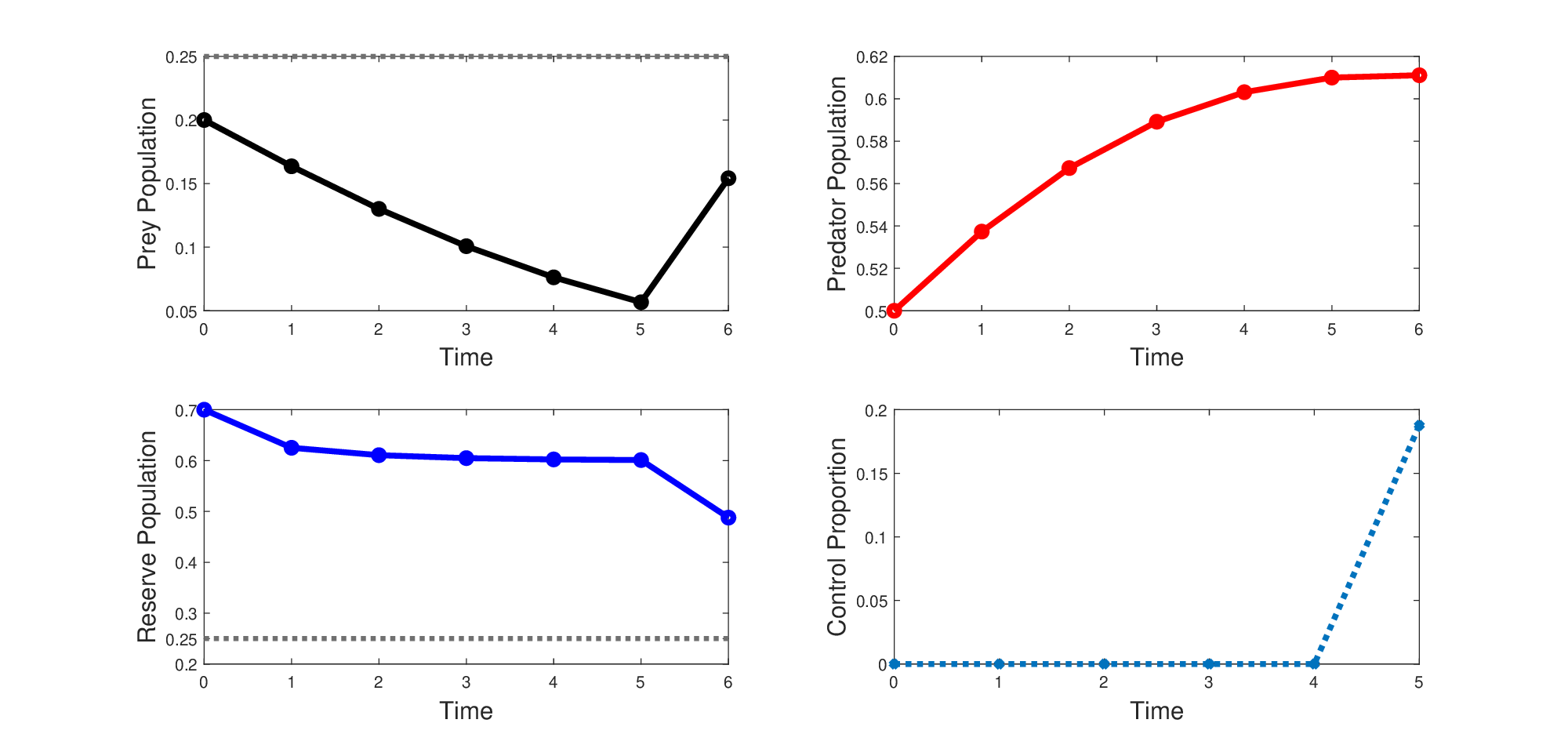}
		\caption{Plots of the states and plot of the optimal control of the discrete augmentation model {\bf Model A} where the population is allowed to grow and then predator-prey action happens before augmentation at each time step using the baseline parameter values (\ref{baselineeqn}) except for $q =0.70, k_w  = 0.60$. The corresponding optimal control and the objective functional values are:  ${\bf h}^\ast = [0, 0, 0, 0, 0, 0, 0.19]$, $J(0) = 0.3418$ (with no control), $J({\bf h}^{\ast}) = 0.3559$ (with optimal control).} \label{Figure4A}
	\end{figure}

   \begin{figure} 
			\centering
			\includegraphics[width=1.1\linewidth]{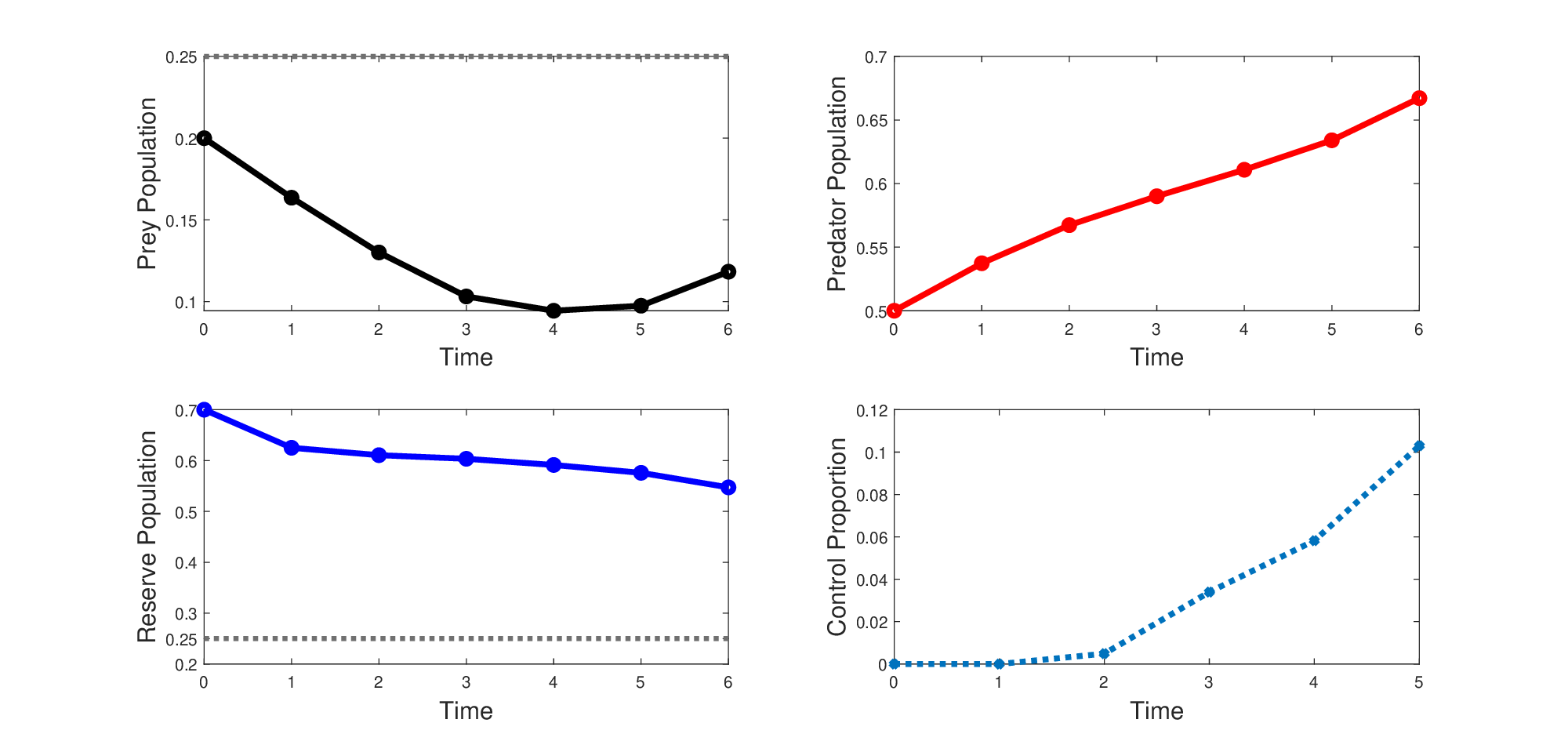}
		\caption{Plots of the states and plot of the optimal control of the discrete augmentation {\bf Model B} where the population is augmented and then grows and then predator-prey action at each time step using the baseline parameter values (\ref{baselineeqn}) except for $q =0.70, k_w  = 0.60$. The corresponding optimal control and the objective functional values are:  ${\bf h}^\ast = [0, 0, 0, 0.03, 0.06, 0.10]$, $J(0) = 0.3418$ (with no control), $J({\bf h}^{\ast}) = 0.3559$ (with optimal control).} \label{Figure4B}
	\end{figure}

      In the simulation of the optimal augmentation {\bf Model A}, the optimal control values, ${\bf h}^\ast = [0, 0, 0, 0, 0, 0.19]$, in Figure~\ref{Figure4A} reveal that augmentation occurs at only the final time step, $t=0.19$, where about $19\%$ of the reserve population is translocated to the target region for augmentation. The $19\%$ represents the value of optimal control at the final time. In this same simulation, the objective functional values  $J(0) = 0.3418$ (with no control) and  $J({\bf h}^{\ast}) = 0.3559$ (with optimal control) give $4\%$ higher objective functional value for implementing the optimal control than when there is no control. Similarly, the simulation of the optimal augmentation {\bf Model B} observed in Figure~\ref{Figure4B} gives the optimal control values  ${\bf h} = [0, 0, 0, 0.03, 0.06, 0.10]$, indicating augmentation occurs at time steps $t = 3, 4$ and $5$, where about $3\%$, $6\%$  and $10\%$ of the reserve population are translocated, respectively.  The corresponding objective functional values $J(0) = 0.3418$ (with no control) and  $J({\bf h}^{\ast}) = 0.3559$ (with optimal control) give $4\%$ higher objective functional value than when there is no optimal control implementation. Though the objective functional values in this scenario's optimal augmentation models are similar, the two augmentation models exhibit different target and reserve population levels at the final time. The results in Figures~\ref{Figure4A} and~\ref{Figure4B} show different final target population levels of $0.15$ and $0.12$, respectively, which are all lower than the minimum threshold for growth $(m=n=0.25)$ as compared to the target population levels in Figures~\ref{Figure1A} and~\ref{Figure1B}. The lower reserve population levels shown in Figures~\ref{Figure4A} and~\ref{Figure4B} are due to the lower reserve species growth rate and carrying capacity.

   \subsection{Effect of the predator mortality rate}
   In this scenario, we increase only the mortality rate, $\gamma$, of the predator population from $0.025$ to $0.10$ and hold all the other parameters in Equation~(\ref{baselineeqn}) fixed. The quantitative results of {\bf Model A} and {\bf Model B} using $\gamma = 0.10$ are shown in Figures~\ref{Figure5A} and~\ref{Figure5B}, respectively.

	\begin{figure} 
			\centering
			\includegraphics[width=1.1\linewidth]{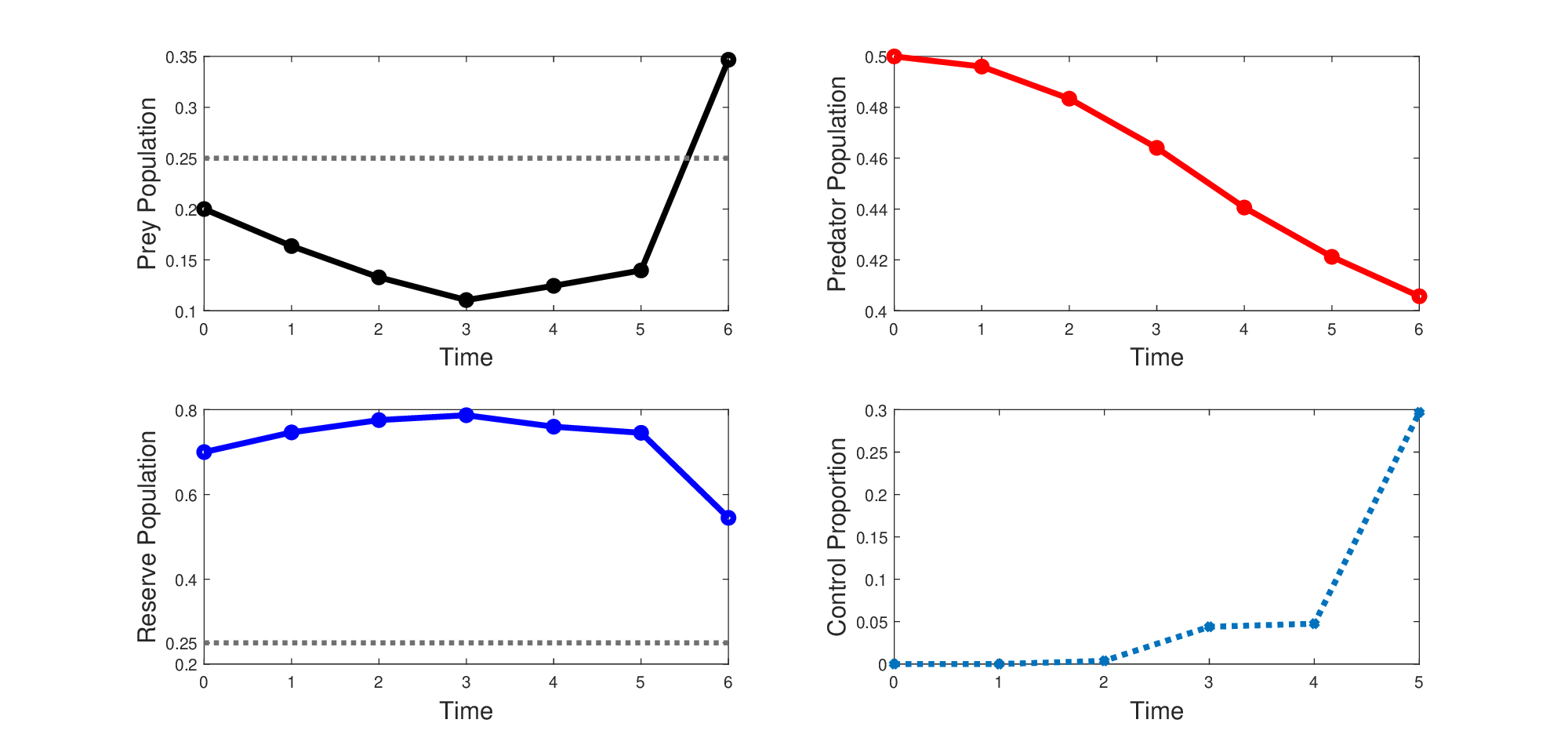}
		\caption{Plots of the states and plot of the optimal control of the discrete augmentation {\bf Model A} where the population is allowed to grow and then predator-prey action happens before augmentation at each time step using the baseline parameter values (\ref{baselineeqn}) except for $\gamma = 0.10$. The corresponding optimal control and the objective functional values are: ${\bf h}^\ast = [0,0,0,0.04,0.05,0.19]$, $J(0) = 0.4572$ (with no control), $J({\bf h}^{\ast}) = 0.5235$ (with optimal control).} \label{Figure5A}
	\end{figure}

	\begin{figure} 
			\centering
			\includegraphics[width=1.1\linewidth]{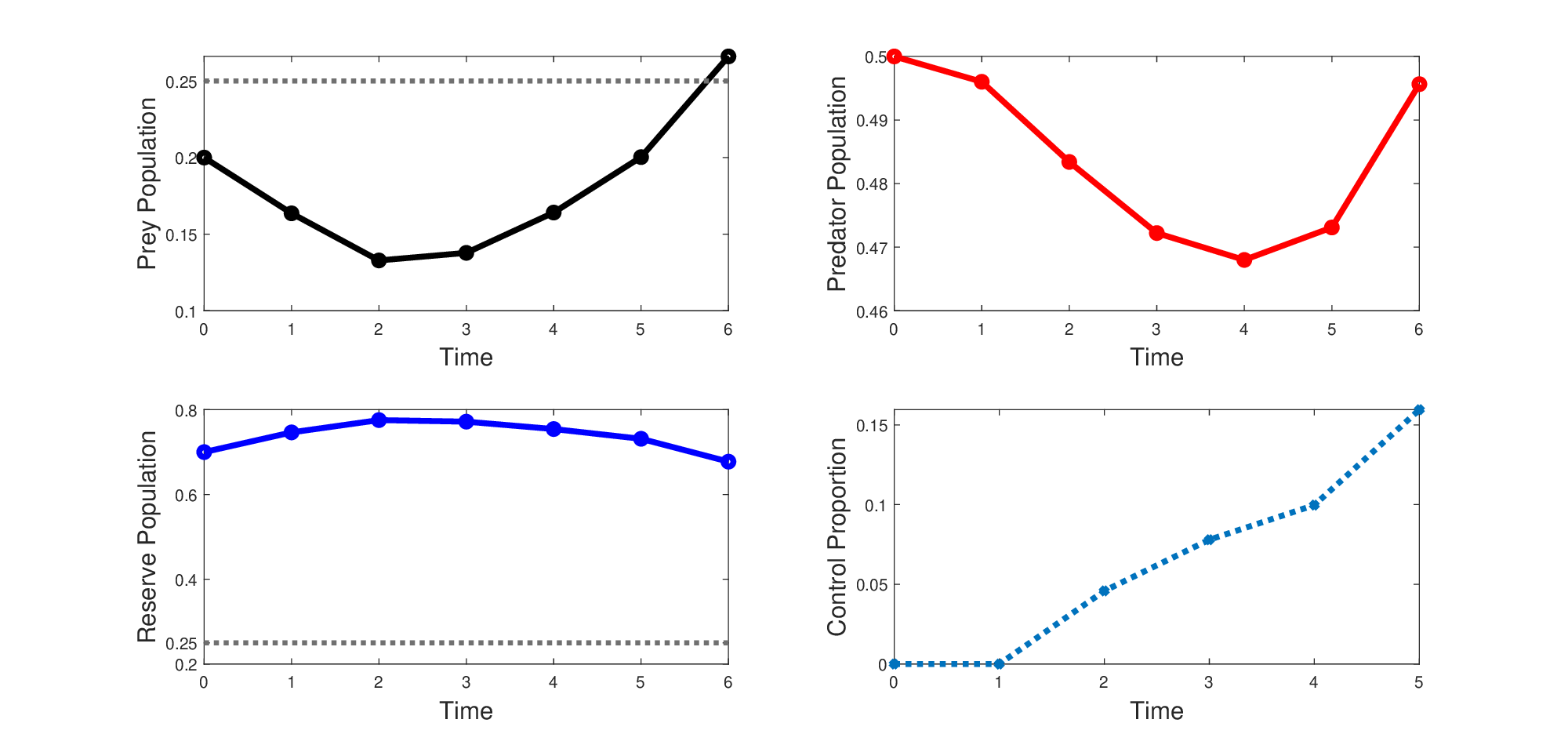}
		\caption{Plots of the states and plot of the optimal control of the discrete augmentation {\bf Model B} where the population is augmented and then grows and then predator-prey action at each time step using the baseline parameter values (\ref{baselineeqn}) except for $\gamma = 0.10$. The corresponding optimal control and the objective functional values are: ${\bf h}^\ast = [0,0,0.05,0.08,0.10,0.16]$, $J(0) = 0.4572$ (with no control), $J({\bf h}^{\ast}) = 0.5299$ (with optimal control).} \label{Figure5B}
	\end{figure}
    
     In Figure~\ref{Figure5A}, the values of optimal controls in optimal augmentation {\bf Model A} for $\gamma = 0.10$   is ${\bf h}^\ast = [0,0,0,0.04,0.05,0.19]$, indicating at each time step $t = 3, 4$ and $5$, about $4\%, 5\%$, and $19\%$ of the reserve are translocated to the target population for augmentation, respectively. This optimal control strategy leads to the target and reserve population level above the threshold for growth at the final time, $T = 6$. The level of the predator population declines towards the end of the period. In this same Figure~\ref{Figure5A}, the prey population level is $0.35$, which is above that of Figure~\ref{Figure1A} when the death rate of the predator population was $0.025$.  The augmentation happens at different time steps of $t = 3, 4$ and $5$. Using the same parameters for the optimal augmentation {\bf Model A}, the objective functional values  $J(0) = 0.4572$ (with no control) and $J({\bf h}^{\ast}) = 0.5235$ (with optimal control) give $15\%$  higher objective functional value with optimal control than when there is no 
     control. 

     In Figure~\ref{Figure5B}, the underlying model is to augment and then grow and then predator-prey action which gives the optimal control values, ${\bf h}^\ast = [0,0,0.05,0.08,0.10,0.16]$, indicating $5\%$, $8\%$, $10\%$ and $16\%$ of the reserve species are translocated to the target region at time steps $t = 2, 3, 4$ and $5$, respectively. Both the target and the reserve population levels in Figure~\ref{Figure5B} are above the threshold for growth; however, the population level of the target species in Figure~\ref{Figure5B} is lower than that of Figure~\ref{Figure5A} using the same set of parameter values.

     The predator population level in {\bf Model B}  declines till the minimum point at $t=4$. The population level then increases again till the final time $t=6$. The augmentation happens at time steps $t = 2, 3, 4$ and $5$. The results of this scenario give the objective functional values  $J(0) = 0.4572$ (with no control) and  $J({\bf h}^{\ast}) = 0.5299$ (with optimal control), representing $16\%$ higher objective functional value with optimal control than no control. 

      It can be observed that the parameter sets used to generate Figures~\ref{Figure5A} and~\ref{Figure5B} lead to qualitatively different dynamics of the predator population when compared to the results generated from the baseline parameter set in Equation~(\ref{baselineeqn}). That is, the predator population is able to rebound in {\bf Model B}, but continues to decline in {\bf Model A}.

 \begin{table}[h!]
\centering
 \caption{Summary of the objective functional values $J({\bf h})$ of the models with and without augmentation, and with percentage increase in the objective functional values for {\bf Model A} and  {\bf Model B} }\label{table1}
\begin{tabular}{lccc}
\hline
Parameter values & No Aug & Model A & Model B \\
\hline
Baseline Eqn~(\ref{baselineeqn}) & 0.4413 & 0.4896 (11\%) & 0.4825 (9\%) \\
Eqn~(\ref{baselineeqn}) with $M_2 = 0$ & 0.4413 & 0.5794 (31\%) & 0.5379 (22\%) \\
Eqn~(\ref{baselineeqn}) with $M_2 = 0$ and $N = 0.1$ & 0.1215 & 0.4662 (284\%) & 0.3178 (162\%) \\
Eqn~(\ref{baselineeqn}) with $q = 0.70$ and $k_w= 0.60$ & 0.3418 & 0.3559 (4\%) & 0.3559(4\%) \\
Eqn~(\ref{baselineeqn}) with $\gamma = 0.10$ & 0.45728 & 0.5235 (15\%) & 0.5299(16\%) \\
\hline
\end{tabular} 
\end{table}

Table~\ref{table1} gives the objective functional values from the simulating Models {\bf A} and {\bf B}. The values without augmentation appear to be the same in both models, irrespective of the parameter scenarios. However, the objective functional values with augmentation differ, except in cases where the growth rate and environmental carrying capacity of the reserve species were reduced, which may have caused a low abundance of the reserve population ready for augmentation.  The percentage values represent the increase in the objective functional values with augmentation compared to the values with no augmentation.

\section{Conclusion}
\label{sec4}
   Discrete-time optimal control theory was applied to a model of species augmentation with predator-prey relationships. The optimal control aims to maximize the prey (target population) and the reserve population at the final time and minimize the associated costs. We carry out the optimal augmentation models for the threatened/endangered species via two orders of events: growth followed by predator-prey action and then augmentation; and augmentation followed by growth and then predator-prey action.  Optimization in the two models is solved numerically using the discrete version of the forward-backward sweep method and the sequential quadratic programming iterative method, respectively. We have indicated the objective functional values for the simulation for each scenario. The simulation results depict different population levels in the 2 models by varying some parameter sets.
  
 In the simulations, with the same baseline parameter values in Figures~\ref{Figure1A} and~\ref{Figure1B}, the optimal augmentation  {\bf Model A} shows $4\%$ and $30\%$ of the reserve population is translocated at time steps $t=4$ and $5$, respectively, suggesting a delayed but intensive intervention strategy. In contrast, the optimal augmentation  {\bf Model B} indicates $5\%$, $8\%$, and $15\%$ of the reserve population are translocated at steps $t =3, 4$,  and $5$, respectively, portraying earlier and gradual translocations of individuals.  These have resulted in an increase in the target population level in Figure~\ref{Figure1A} to $0.3$ at the final time step $T=6$, which is above the minimum threshold for growth, and the target population level in Figure~\ref{Figure1B} is $0.19$, which is below the threshold for growth.  Again, decreasing the reserve population growth rate from $0.85$ to $0.70$ and the reserve population's environmental carrying capacity from $0.80$ to $0.60$ results in the target population level falling below the growth threshold in the 2 models by the final time.  This can be observed in Figures~\ref{Figure4A} and~\ref{Figure4B} with different target population levels of $0.15$ and $0.12$, respectively, which are all lower than the minimum threshold for growth. This means the reserve population should grow well initially to have enough individuals for translocation.  Furthermore,  the predator population's death rate plays an important role in the dynamics of each order of events as observed in Figures~\ref{Figure5A} and~\ref{Figure5B}. The reserve and the target population sizes of each optimal augmentation strategy are above the threshold for growth at the final time when the predator mortality rate is increased from $0.025$ to $0.10$, which can be observed in Figures~\ref{Figure5A} and~\ref{Figure5B}, respectively.  The level of the predator population declines towards the end of the final time in Figure~\ref{Figure5A} since more individuals die naturally. 
   
   In summary, we have illustrated the use of optimal control theory for species augmentation in predator-prey dynamics. The study is conducted using a discrete-time difference equation. We considered two optimal augmentation models with different orders of events. The results of the simulations of the two models show different population levels at the end of the augmentation horizon using the same parameter sets, indicating the importance of using a different order of events. For instance, natural resource managers interested in maximizing the total population at the end of the augmentation horizon may opt for the event that yields higher population levels by the final time. This work is the first optimal augmentation model incorporating the predator-prey relationship with discrete equations.
   
  There are other possible orders of events that can be taken into consideration in our model and are yet to be explored. One may want to consider a scenario where the predator-prey interaction occurs followed by the growth of the prey population but before the decay of the predator population, or prior to both the growth of the prey population and decay of the predator population and then augmentation. These and many more possible orders of events can be explored, and we recommend them for future extensions.

  \section*{Data Availability}
  No datasets were generated or analyzed during the current study.
     
   \section*{Acknowledgements}
The authors are grateful to the Pan African University Institute for Basic Sciences, Technology and Innovation (PAUSTI) for the support.

\bibliographystyle{unsrtnat}
\bibliography{References}

\end{document}